\documentclass[12pt]{amsart}

\usepackage{amsmath} 
\usepackage{amsthm} 
\usepackage{amssymb} 

\usepackage{cite}
\makeatletter
\newcommand{\citecomment}[2][]{\citen{#2}#1\citevar}
\newcommand{\citeone}[1]{\citecomment{#1}}
\newcommand{\citetwo}[2][]{\citecomment[,~#1]{#2}}
\newcommand{\citevar}{\@ifnextchar\bgroup{;~\citeone}{\@ifnextchar[{;~\citetwo}{]}}}
\newcommand{\citefirst}{\@ifnextchar\bgroup{\citeone}{\@ifnextchar[{\citetwo}{]}}}

\makeatother

\usepackage{microtype} 
\usepackage{pinlabel} 

\newcommand{\sqdiamond}{\tikz [x=1.2ex,y=1.85ex,line width=.1ex,line join=round, yshift=-0.285ex] \draw  (0,.5) -- (0.75,1) -- (1.5,.5) -- (0.75,0) -- (0,.5) -- cycle;}%
\renewcommand{\Diamond}{$\sqdiamond$} 

\usepackage[scaled=0.9]{sourcecodepro} 

\usepackage{enumitem}
\usepackage{color}
\usepackage{subfig, caption}
 
\usepackage{wrapfig}
\usepackage{pdflscape}
\usepackage{array}
\usepackage{tikz-cd}
\usepackage{longtable, booktabs}

\usepackage[hidelinks, pagebackref]{hyperref}

\makeatletter
\define@key{href}{font}{#1}
\makeatother
\usepackage{xpatch}
\newcommand\hrefdefaultfont{\ttfamily}
\xpatchcmd\href{\setkeys{href}{#1}}{\setkeys{href}{font=\hrefdefaultfont,#1}}{}{\fail}

\renewcommand*{\backref}[1]{}
\renewcommand*{\backrefalt}[4]{
  \ifcase #1 
  [No citations.]
  \or [#2]
  \else [#2]
  \fi }

\let\originalleft\left
\let\originalright\right
\renewcommand{\left}{\mathopen{}\mathclose\bgroup\originalleft}
\renewcommand{\right}{\aftergroup\egroup\originalright}

\newcommand{\calF}{\mathcal{F}}
\newcommand{\calG}{\mathcal{G}}
\newcommand{\calH}{\mathcal{H}}

\newcommand{\calL}{\mathcal{L}}

\newcommand{\calN}{\mathcal{N}}

\newcommand{\calT}{\mathcal{T}}

\newcommand{\TT}{\mathbb{T}}

\newcommand{\from}{\colon} 
 
\newcommand{\cross}{\times}

\newcommand{\cover}[1]{{\widetilde{#1}}}

\newcommand{\bdy}{\partial} 

\newcommand{\interior}{{\operatorname{interior}}}

\input{header_article.tex}

\theoremstyle{plain}
\newtheorem{XXXtheoremQED}[equation]{Theorem} 
  {\pushQED{\qed}\begin{XXXtheoremQED}}
  {\popQED\end{XXXtheoremQED}}

\newcommand{\fakeenv}{} 

\newenvironment{restate}[2]  
{ 
 \renewcommand{\fakeenv}{#2} 
 \theoremstyle{plain} 
 \newtheorem*{\fakeenv}{#1~\ref{#2}} 
 \begin{\fakeenv}
}
{
 \end{\fakeenv}
}

\newenvironment{restated}[2]  
{ 
 \renewcommand{\fakeenv}{#2} 
 \theoremstyle{definition} 
 \newtheorem*{\fakeenv}{#1~\ref{#2}} 
 \begin{\fakeenv}
}
{
 \end{\fakeenv}
}

\newcommand{\Nature}{\operatorname{\mathsf{Nat}}}
\newcommand{\Visited}{\operatorname{\mathsf{Vis}}}

\newcommand{\phidelta}{\delta}  
\newcommand{\phif}{f} 

\usepackage{tikz-cd}

\title[Connecting essential triangulations]{Connecting essential triangulations II: \\ via 2-3 moves only}
\author[Kalelkar, Schleimer, Segerman]{Tejas Kalelkar, Saul Schleimer, and Henry Segerman} 
\date{\today}

\begin{document}

\begin{abstract}
In previous work we showed that for a manifold $M$, whose universal cover has infinitely many boundary components, the set of essential ideal triangulations of $M$ is connected via 2-3, 3-2, 0-2, and 2-0 moves.
Here we show that this set is also connected via 2-3 and 3-2 moves alone, if we ignore those triangulations for which no 2-3 move preserves essentiality.
If we also allow V-moves and their inverses then the full set of essential ideal triangulations of $M$ is once again connected.
These results also hold if we replace essential triangulations with $L$--essential triangulations.
\end{abstract}




\maketitle

\section{Introduction}

\subsection{Graphs of triangulations}

Combinatorial moves on triangulations have been studied extensively. 
A particularly elegant formulation was given by Pachner~\cite{Pachner78}, who showed that bistellar moves connect any two triangulations of a given manifold.
For a three-dimensional manifold $M$, these are the 2-3, 3-2, 1-4, and 4-1 moves.
If we do not want to change the number of vertices, we may restrict ourselves to the former two moves by the work of 
Matveev~\cite[Theorem~1.2.5]{Matveev07}, Piergallini~\cite[Theorem~1.2]{Piergallini88}, and Amendola~\cite[Theorem~2.1]{Amendola05}.
To state their results, let $\TT(M)$ be the set of one-vertex triangulations (if $M$ is closed) \emph{or} the set of ideal triangulations (if $M$ has boundary).
Let $\TT_2(M)$ be those triangulations of $\TT(M)$ having at least two tetrahedra.

\begin{figure}[htbp]
\includegraphics[width = \textwidth]{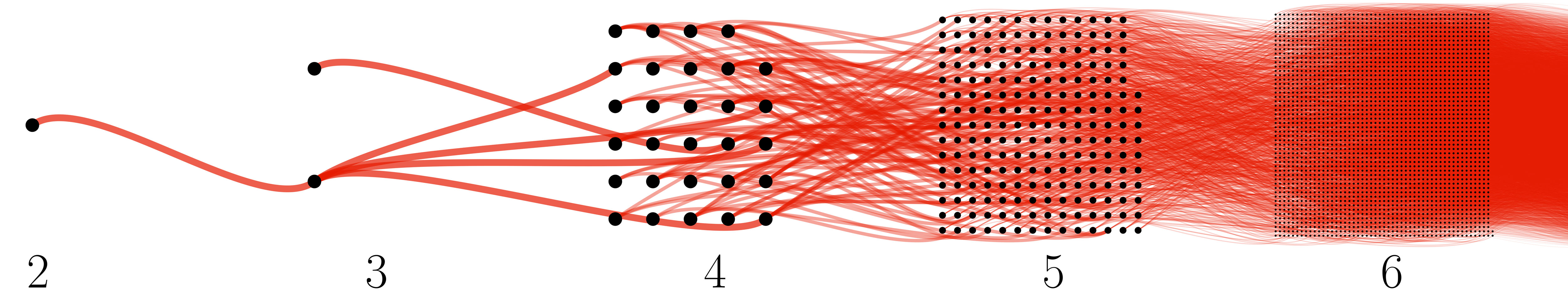}
\caption{Part of the graph having $\TT_2(S^2 \cross S^1)$ as vertices (triangulations arranged by number of tetrahedra) and 2-3 moves as edges.  
This and the following graphs were generated using Regina~\cite{regina}.}
\label{Fig:AllTri}
\end{figure}

\begin{theorem}[Matveev--Piergallini--Amendola]
\label{Thm:MPA}
Suppose that $M$ is a three-manifold.
Then $\TT_2(M)$ is connected via 2-3 and 3-2 moves. \qed
\end{theorem}


We illustrate this for $S^2 \times S^1$ in \reffig{AllTri}.
An \emph{essential} ideal triangulation has the property that none of its edges are homotopic, relative to their endpoints, into $\bdy M$.
(When $M$ is closed, we require instead that none of its edges are null-homotopic. Equivalently, by removing regular open neighbourhoods of the vertices, we may convert material triangulations into ideal ones.)
Essentiality is required in many applications.
For example, an ideal triangulation of a cusped hyperbolic three-manifold is essential if and only if it admits a solution to Thurston's gluing equations~\cite[Theorem~1]{SegermanTillmann11}.
(Here we do not assume that all tetrahedra are positively oriented.)
Thus the 1-loop invariant of Dimofte and Garoufalidis~\cite[Definition~1.2]{DimofteGaroufalidis13} is only defined on essential triangulations, as is the Bloch invariant~\cite[Definition~2.5]{NeumannYang99}.
Moreover, a triangulation admitting a strict angle structure or a taut angle structure is necessarily essential~\cite[Theorem~6.1]{HodgsonRubinsteinSegermanTillmann15}.
Unfortunately, the set of essential triangulations need not be connected via 2-3 and 3-2 moves.
This is the case for $S^2 \cross S^1$, as shown in \reffig{EssTriDisconn}.

\begin{figure}[htbp]
\includegraphics[width = \textwidth]{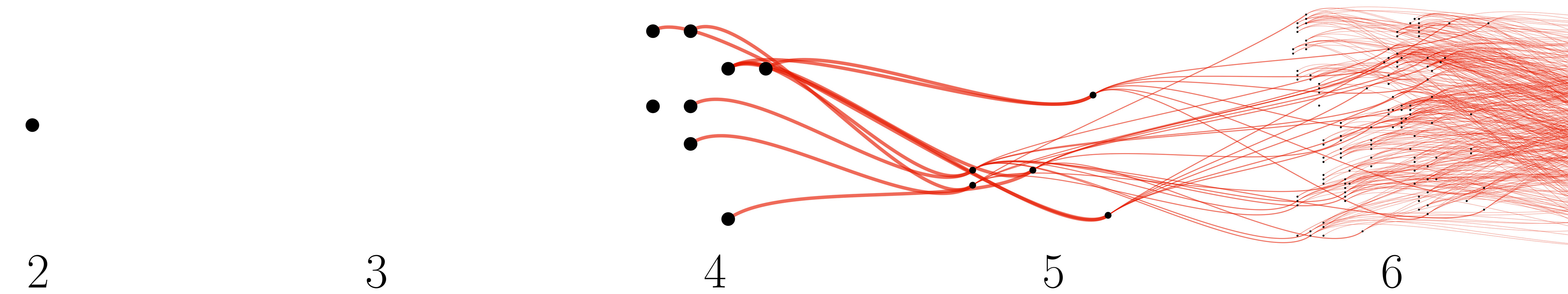}
\caption{The subgraph of \reffig{AllTri} spanned by essential triangulations.
For a discussion of the isolated essential triangulation with two tetrahedra, 
see \refexa{S2xS1}.}
\label{Fig:EssTriDisconn}
\end{figure}

In the prequel~\cite{KSS24a}  to this paper we prove a connectivity result similar to \refthm{MPA} but where the initial and terminal triangulations, as well as all intermediate triangulations, are essential. 
To do this, we use \emph{0-2 and 2-0 moves}; see \refsec{0-2}.
The resulting connected graph (again for $S^2 \cross S^1$) is shown in \reffig{EssTriConn}.

\begin{figure}[htbp]
\includegraphics[width = \textwidth]{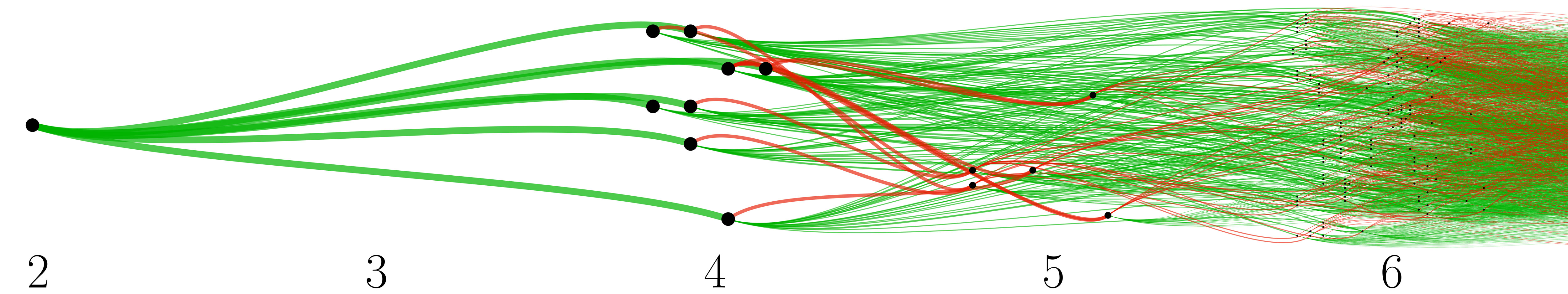}
\caption{The graph of \reffig{EssTriDisconn} with additional edges for 0-2 moves.}
\label{Fig:EssTriConn}
\end{figure}

Also in that paper we generalise to \emph{$L$--essential} triangulations.
These involve a choice of \emph{labelling} $L$ of the boundary components $\Delta_M$ of the universal cover of the manifold $M$. 
See \refdef{LEssential} for the details. 
Let $\TT(M, L)$ be those (necessarily ideal) triangulations in $\TT(M)$ which are $L$--essential.
The full result~\cite[Theorem~6.1]{KSS24a} is as follows. 

\begin{theorem}
\label{Thm:ConnectivityWith0-2}
Suppose that $M$ is a compact, connected three-manifold with boundary.
Suppose that $L$ is a labelling of $\Delta_M$ with infinite image.
Then $\TT(M, L)$
is connected via 2-3, 3-2, 0-2, and 2-0 moves. \qed
\end{theorem}

Let $\TT^\circ(M, L)$ be the triangulations in $\TT(M, L)$ which admit some 2-3 or 3-2 move preserving $L$--essentiality.
The main goal of this paper is to prove the following.

\begin{restate}{Theorem}{Thm:Connectivity}
Suppose that $M$ is a compact, connected three-manifold with boundary.
Suppose that $L$ is a labelling of $\Delta_M$ with infinite image.
Then $\TT^\circ(M, L)$ is connected via 2-3 and 3-2 moves.
\end{restate}

This implies that in the graph illustrated in \reffig{EssTriDisconn} there is only one component with edges.
We say that a triangulation is \emph{isolated} if it lies in $\TT(M, L) - \TT^\circ(M, L)$.

As it happens, every isolated triangulation can be connected to $\TT^\circ(M, L)$ by a single \emph{V-move}, a more ``local'' 0-2 move.
See \refdef{VMove} and \reflem{IsolatedVMove}. 
Using (selected) V-moves instead of the more general 0-2 moves yields the connected graph shown in \reffig{EssTriConnV}.

\begin{figure}[htbp]
\includegraphics[width = \textwidth]{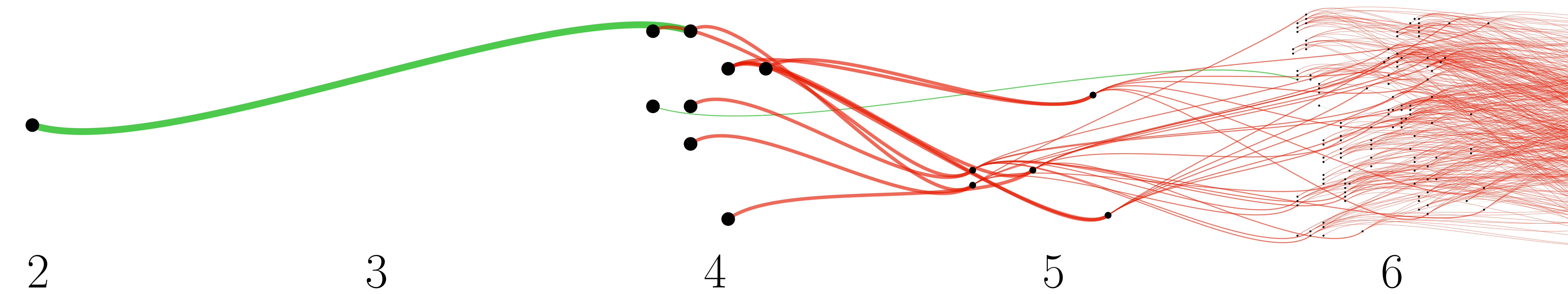}
\caption{Essential one-vertex triangulations of $S^2 \cross S^1$, as connected by 2-3 and
(selected) V-moves.}
\label{Fig:EssTriConnV}
\end{figure}

Combining \refthm{Connectivity} and \reflem{IsolatedVMove}, we obtain the following improvement of \refthm{ConnectivityWith0-2}.

\begin{corollary}
\label{Cor:ConnectivityVMove}
Suppose that $M$ is a compact, connected three-manifold with boundary.
Suppose that $L$ is a labelling of $\Delta_M$ with infinite image.
Then $\TT(M, L)$ is connected via 2-3 moves, V-moves, and their inverse moves. \qed
\end{corollary}

In~\cite{SchleimerSegerman24Finiteness}, the second and third authors will use \refcor{ConnectivityVMove} to classify the veering triangulations of hyperbolic once-punctured torus bundles and various other manifolds.
To do this, we carry a ``winding'' and a compatible circular ordering through a sequence of essential triangulations, connected by 2-3 moves, V-moves, and their inverses.
Surprisingly, some windings cannot be carried through certain 2-0 moves. 
Thus \refthm{ConnectivityWith0-2} does not suffice for this application. 

\subsection{Outline}

Given \refthm{ConnectivityWith0-2}, it suffices to be able to implement a 0-2 move between $L$--essential triangulations via a sequence of 2-3 and 3-2 moves where all intermediate triangulations are also $L$--essential. 
In the absence of a requirement of $L$--essentiality, this can be done purely locally.
See Lemma~1.2.11 and Proposition~1.2.8 of~\cite{Matveev07}.

In \refsec{FollowingMatveev}, we lay out the hypotheses and tools needed to make this local construction go through with the additional requirement of $L$--essentiality.
Our versions of the local construction are set out in Lemmas~\ref{Lem:Do0-2ThreeSides} and~\ref{Lem:Do0-2ManyEdges}.
However, as illustrated in \refsec{LocallyFrozen}, it may be that near the site of the 0-2 move there are no 3-2 moves, and any 2-3 move destroys $L$--essentiality.
That is, any nearby 2-3 move introduces an edge between two vertices in the universal cover with the same label. 
Thus we must bring some ``distant'' vertex, with a different label, ``close'' to the site of the 0-2 move.

In \refsec{DistantLabels} we give a collection of moves (the \emph{augmented 2-3 move} and the \emph{nature reserve moves}) that preserve $L$--essentiality and serve as tools to ``transport'' the distant vertex.
In \refsec{ParallelSequences}, we assemble these moves into two parallel sequences, appropriately commuting with the 0-2 move.
These sequences move the distant vertex into contact with the site of the 0-2 move, at which point we apply \reflem{Do0-2ManyEdges}.

\subsection*{Acknowledgements}

The third author was supported in part by National Science Foundation grant DMS-2203993.

\section{Background}

Here we essentially follow the notation and definitions of~\cite[Section~2]{KSS24a}.

\subsection{Triangulations and foams}

Our main theorem, \refthm{Connectivity}, is stated in terms of triangulations.
However, their dual \emph{foams} are often easier to understand.
(One possible explanation for this is that the complexity concentrated at the vertices of a triangulation is spread across the boundaries of the three-cells in the dual foam. 
For a simple example, see \reffig{HypothesisExamples}.)
Both triangulations and foams are complexes formed by taking a disjoint union of cells and gluing them together.
We refer to the cells, before gluing, as \emph{model} cells.

\begin{definition}
A \emph{triangulation} $\calT$ is a collection of model tetrahedra $\{t_k\}$ together with a collection of \emph{face pairings} $\{\phi_{ij}\}$.
Here $\phi_{ij}$ is an isomorphism from some model face $f_i$ of some model tetrahedron $t_{k(i)}$ to some model face $f_j$ of some model tetrahedron $t_{k(j)}$.
The \emph{realisation} of $\calT$ is the topological space $|\calT|$ obtained by taking the disjoint union of the $t_i$ and forming the quotient by the $\phi_{ij}$.
The zero-skeleton of $|\calT|$ is the image of the model vertices.

Suppose that $M$ is a compact, connected three-manifold with boundary.
A triangulation $\calT$ is an \emph{ideal triangulation} of $M$ if $M - \bdy M$ is homeomorphic to $|\calT|$ minus its zero-skeleton.
\end{definition}

\begin{definition}
We denote the universal covering map by $\phi_M \from \cover{M} \to M$.
We use $\Delta_M$ to denote the set of boundary components of $\cover{M}$.
\end{definition}

\begin{definition}
\label{Def:Foam}
Suppose that $M$ is a compact, connected three-manifold, with boundary.
Suppose that $\calT$ is an ideal triangulation of $M$.
We call $\calF$, the dual two-complex to $\calT$, a \emph{foam} in $M$.
We refer to the components of $M - \calF$ as \emph{(complementary) regions}.
\end{definition}

See Figures~\ref{Fig:PointTypesOnFoam1},~\ref{Fig:PointTypesOnFoam2}, and~\ref{Fig:PointTypesOnFoam3} for small neighbourhoods of points of $\calF$ in $M$.

\begin{figure}[htbp]
\subfloat[]{
\includegraphics[height = 2.4cm]{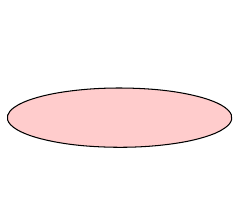}
\label{Fig:PointTypesOnFoam1}
}
\quad
\subfloat[]{
\includegraphics[height = 2.4cm]{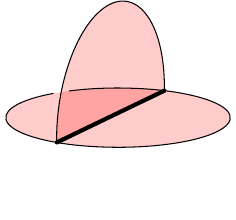}
\label{Fig:PointTypesOnFoam2}
}
\quad
\subfloat[]{
\includegraphics[height = 2.4cm]{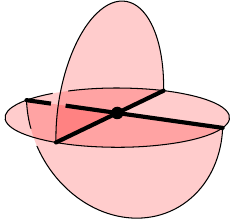}
\label{Fig:PointTypesOnFoam3}
}
\quad
\subfloat[]{
\includegraphics[height = 2.4cm]{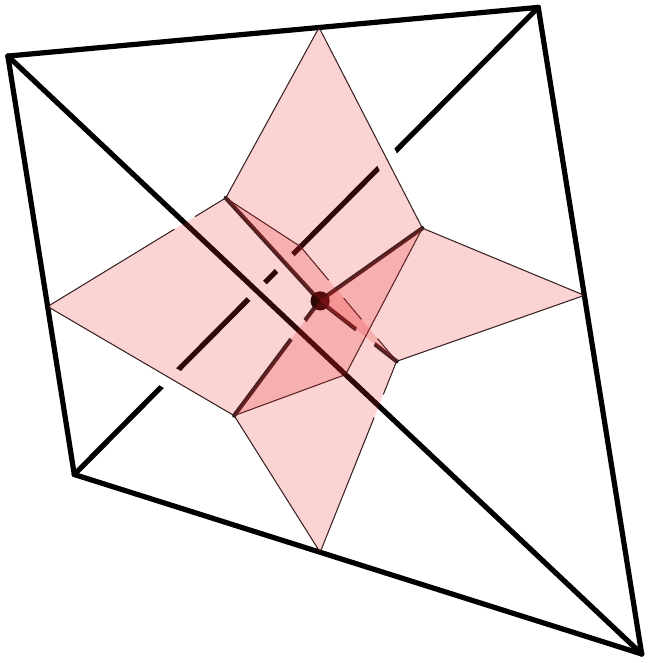}
\label{Fig:Butterfly}
}
\caption{Local pictures of foams.}
\label{Fig:Foam}
\end{figure}

\begin{definition}
\label{Def:EdgeLoop}
An edge $e$ of a foam is an \emph{edge loop} if $e$ has both ends at a single vertex.
Any lift of $e$ to a cover (including the trivial cover) is a \emph{cyclic edge}.
\end{definition}  

For examples of edge loops and cyclic edges see \reffig{HypothesisExamples}.

\subsection{Labellings and $L$--essentiality}
\label{Sec:Labellings}

\begin{definition}
\label{Def:Labelling}
Suppose that $\calL$ is a set of \emph{labels} equipped with an action of $\pi_1(M)$.
Suppose that $L \from \Delta_M \to \calL$ is a $\pi_1(M)$--equivariant function.
Then we call $L$ a \emph{labelling} of $\Delta_M$.
\end{definition}

\begin{definition}
\label{Def:LEssential}
Suppose that $\calT$ is an ideal triangulation of $M$.
Suppose that $L$ is a labelling of $\Delta_M$, as in \refdef{Labelling}.
Suppose that $e$ is an edge of $\calT$ with a lift $\cover{e}$ in $\cover{\calT}$.
Suppose that $\cover{u}$ and $\cover{v}$ are the endpoints of $\cover{e}$.
If $L(\cover{u}) \neq L(\cover{v})$ then we say that $e$ is \emph{$L$--essential}.
If all edges of $\calT$ are $L$--essential then we say that $\calT$ is \emph{$L$--essential}.
\end{definition}

See \cite[Section~2.9]{KSS24a} for examples of labellings.
The simplest labelling is the identity map on $\Delta_M$.
As noted in \cite[Remark~2.18]{KSS24a}, with this labelling $L$--essential triangulations are essential triangulations in the sense of \cite[Definitions~3.2 and~3.5]{HodgsonRubinsteinSegermanTillmann15} and \cite[page~336]{LuoTillmannYang13}.

Dually, our notions of $L$--essentiality apply to foams as follows.

\begin{definition}
\label{Def:LEssentialFoam}
Suppose that $\calF$ is a foam in $M$.
We extend the labelling function $L$ to components of $\cover{M} - \cover{\calF}$ as follows.
Suppose that $C$ is a component of $\cover{M} - \cover{\calF}$ with boundary component $c \in \Delta_M$.
Then we set $L(C) = L(c)$.

Now suppose that $f$ is a face of a foam $\calF$ with a lift $\cover{f}$ in $\cover{\calF}$, with components $U$ and $V$ of $\cover{M} - \cover{\calF}$ incident to $\cover{f}$.
We say that $f$ is \emph{$L$--essential} if $L(U) \neq L(V)$.
If all faces of $\calF$ are $L$--essential then we say that $\calF$ is \emph{$L$--essential}.
\end{definition}

\subsection{Moves on foams}

The three-dimensional bistellar moves~\cite{Pachner78} are the 1-4, 2-3, 3-2, and 4-1 moves.
These can be performed equally well on triangulations or their dual foams.
Of these we only consider the 2-3 and 3-2 moves, applied to foams.
The former is called the \emph{T move} by Matveev~\cite[page~14]{Matveev07}.
It can be performed along any edge of $\calF$ that is not an edge loop. 
See \reffig{2-3}.
The 3-2 move can be performed on any triangular face whose closure is embedded in $\calF$.

\begin{figure}[htbp]
\subfloat[]{
\includegraphics[height = 4.cm]{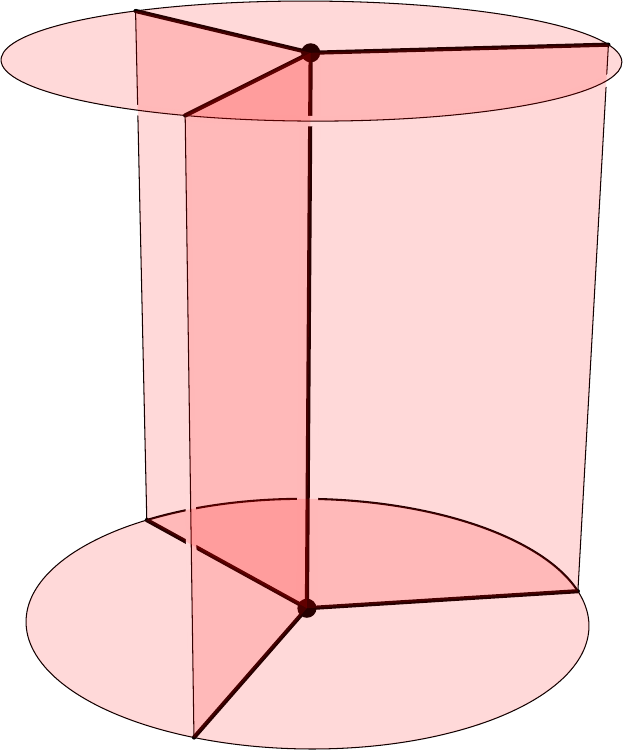}
\label{Fig:2-3Before}
}
\qquad
\subfloat[]{
\includegraphics[height = 4.cm]{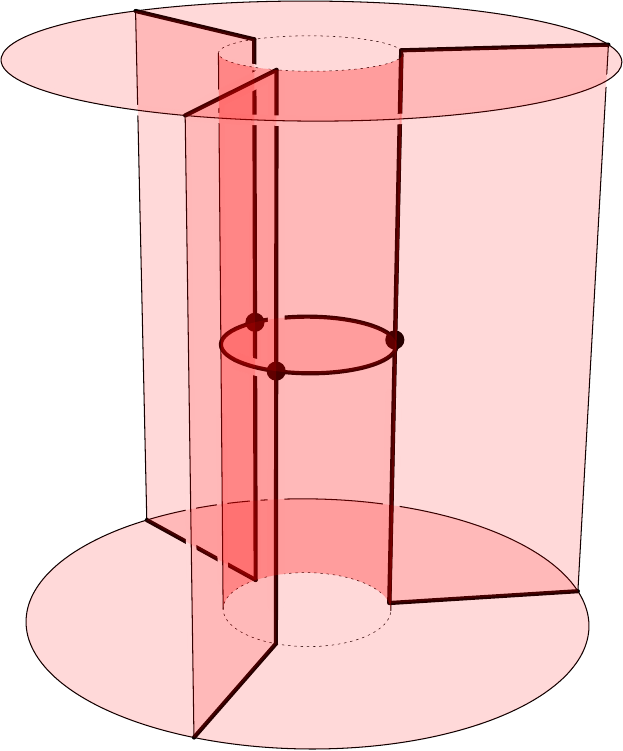}
\label{Fig:2-3After}
}
\caption{The 2-3 move.}
\label{Fig:2-3}
\end{figure}

\subsubsection{The 0-2 move}
\label{Sec:0-2}

The \emph{0-2 move}, defined by \reffig{0-2}, is called the \emph{ambient lune} move by Matveev~\cite[page~17]{Matveev07}.
It is applied along an arc $\delta$ properly embedded in, and avoiding the vertices of, a face of $\calF$.
The 0-2 move creates two new vertices and a new bigon face.
We denote the result by $\calF[\delta]$.

\begin{figure}[htbp]
\subfloat[]{
\labellist
\hair 2pt \small
\pinlabel $\delta$ [r] at 144 200
\endlabellist
\includegraphics[height = 4cm]{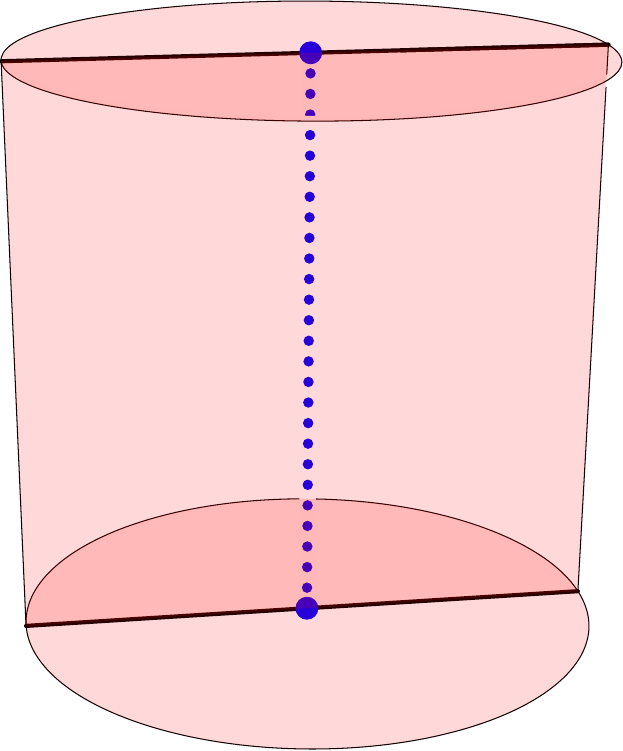}
\label{Fig:0-2Before}
}
\qquad
\subfloat[]{
\labellist
\hair 2pt \small
\pinlabel $f$ at 58 195
\pinlabel $f'$ at 242 200
\endlabellist
\includegraphics[height = 4cm]{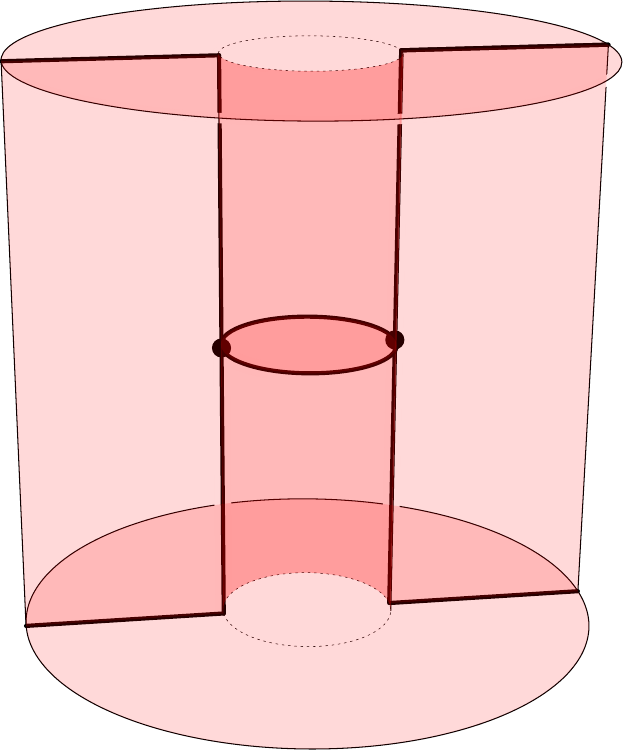}
\label{Fig:0-2After}
}
\caption{The 0-2 move. The vertical dotted arc $\delta$ in \reffig{0-2Before} indicates the arc along which the 0-2 move acts.}
\label{Fig:0-2}
\end{figure}

\begin{definition}
\label{Def:Ancestor}
Suppose that $e$ is an edge of a foam $\calF$.
Let $\calN(e)$ be a small regular neighbourhood of $e$.
Suppose that $\calF'$ is the result of applying a 2-3 move to $\calF$ along $e$.
We assume that the move is supported in $\calN(e)$.
Suppose that $c$ is an open cell (or open complementary region) of $\calF$.
Suppose that $c'$ is similarly obtained from $\calF'$.
If $c \cap c' - \calN(e)$ is non-empty then we say that $c$ the \emph{ancestor} of $c'$ and $c'$ is the \emph{descendant} of $c$.

We make similar definitions for the 3-2, 0-2, 2-0, and various other moves we define later in the paper.
Finally we make the relation transitive through multiple moves.
\end{definition}

\begin{remark}
In a small abuse of notation we often use the same name for an ancestor and its descendants.
\end{remark}

\subsection{Edge loops cause issues}

Here we give an explicit example of an isolated essential triangulation: that is, no 2-3 or 3-2 move preserves essentiality.

\begin{figure}[htbp]
\subfloat[A foam in $S^1 \times D^2$.]{
\includegraphics[height = 4.0cm]{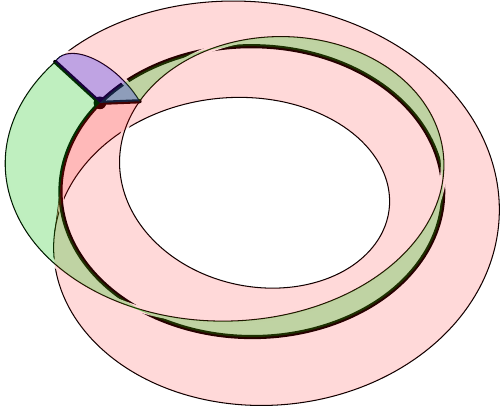}
\label{Fig:S2xS1Example}
}
\quad
\subfloat[The dual triangulation. The internal triangle is shaded.]{
\includegraphics[height = 4.0cm]{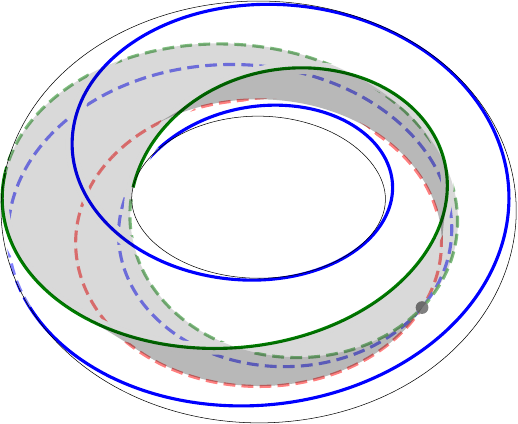}
\label{Fig:SolidTorusTetrahedron}
}

\subfloat[Part of the universal cover for the foam in \reffig{S2xS1Example}.]{
\includegraphics[height = 4.0cm]{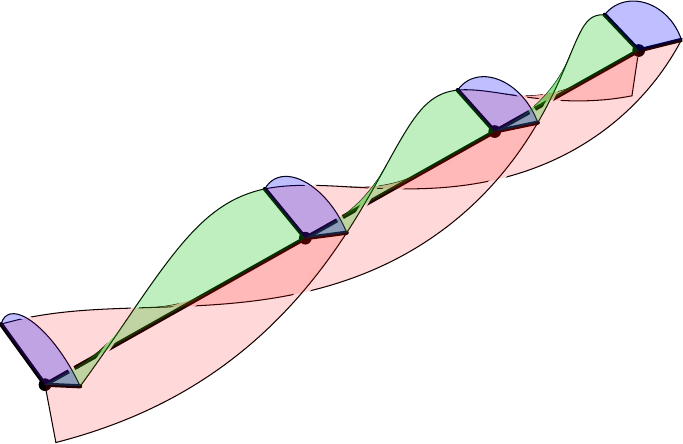}
\label{Fig:S2xS1ExampleCover}
}
\caption{We obtain a foam in $S^1 \times S^2$ by doubling the foam in \reffig{S2xS1Example} across the boundary of the solid torus.
There is an edge loop in \reffig{S2xS1Example}.
Its lifts in \reffig{S2xS1ExampleCover} are cyclic edges. }
\label{Fig:HypothesisExamples}
\end{figure}

\begin{example}
\label{Exa:S2xS1}
\reffig{S2xS1Example} shows a foam in a solid torus.
Mirroring this foam across the boundary torus produces a foam $\calF$ in $S^2 \times S^1$ with two vertices, four edges, three faces, and one complementary region. 
Let $L$ be the identity labelling (as described in \refsec{Labellings}).
All faces of $\calF$ are essential (in the universal cover the regions to either side of each face are distinct).
There are no triangular faces so no 3-2 move is possible. 
Two of the edges of $\calF$ are edge loops so a 2-3 move cannot be applied along them. 
The other two edges bound a bigon face, which implies that a 2-3 move applied along them introduces an inessential face.
It follows that the triangulation $\calT$ dual to $\calF$ is isolated.
In fact, $\calT$ is the unique triangulation with two tetrahedra in \reffig{EssTriDisconn}.
\end{example}

This example is similar in spirit to the finite number of exceptions to Matveev's, Piergallini's, and Amendola's original result (\refthm{MPA}).
There, a triangulation with a single tetrahedron is ``isolated''.
However in our context, starting with $\calF$ and repeatedly applying \refrem{MoreIsolated} produces infinitely many examples. 

\section{$L$--flippable edges and cyclic edges}

Suppose that $M$ is a compact, connected three-manifold with boundary.
Suppose that $L$ is a labelling of $\Delta_M$.

\begin{definition}
\label{Def:LFlippable}
Suppose that $\calT$ is an $L$--essential ideal triangulation of $M$.
Suppose that $f$ is a face of $\calT$ and suppose that performing a 2-3 move across $f$ produces $\calT'$, which is also $L$--essential.
Then we say that $f$ is \emph{$L$--flippable}.
We make the analogous definition for edges of foams.
\end{definition}

\begin{remark}
Suppose that $\calT$ is an $L$--essential ideal triangulation.
Suppose that $\calT$ has a 3-2 move along the edge $e$ that preserves $L$--essentiality.
Then, for $f$ any face adjacent to $e$, the 2-3 move across $f$ also preserves $L$--essentiality. 
\end{remark}

Therefore, in the graph with $\TT(M,L)$ for vertices and 2-3 moves for edges, 
a triangulation is isolated if and only if it has no $L$--flippable faces.

\begin{definition}
\label{Def:NonIsolated}
Let $\TT^\circ(M,L)$ be the set of $L$--essential triangulations of $M$ that have at least one $L$--flippable face.
\end{definition}

\begin{remark}
If $L$ has infinite image then $\TT(M,L)$ is non-empty by \cite[Theorem~3.1]{KSS24a}.
\reflem{IsolatedVMove} then implies that $\TT^\circ(M,L)$ is non-empty.
\end{remark}

Our main result is the following.

\begin{theorem}
\label{Thm:Connectivity}
Suppose that $M$ is a compact, connected three-manifold with boundary.
Suppose that $L$ is a labelling of $\Delta_M$ with infinite image.
Then the set $\TT^\circ(M,L)$ is connected via 2-3 and 3-2 moves.
\end{theorem}

To prove this we require several tools.

\subsection{Finding $L$-flippable edges} 

\begin{lemma}
\label{Lem:LInfinite5Labels}
Suppose that the labelling $L$ has infinite image.
Then there is an edge $e$ of $\cover{\calF}$ that is incident to complementary regions with five distinct labels. 
\end{lemma}

\begin{proof}
We prove the contrapositive. 
Suppose that $e$ is an edge of $\cover{\calF}$ with endpoints $v$ and $w$.
We assume that $e$ is incident to at most four distinct labels.
Let $A$, $B$, and $C$ be the labels of the complementary regions meeting the interior of $e$.
Let $D$ and $E$ be the labels of the complementary regions meeting $v$ and $w$ but not meeting the interior of $e$.
The labels $A$, $B$, and $C$ are distinct because $\calF$ is $L$--essential.
Similarly, the labels $D$ and $E$ are each distinct from $A$, $B$, and $C$.
Thus we must have that $D = E$.

The one-skeleton of $\cover{\calF}$ is connected, so propagating the above argument we find that every vertex of $\cover{\calF}$ is incident to regions with labels $A$, $B$, $C$, and $D$.
Thus $|L(\Delta_M)| = 4$.
\end{proof}

\begin{lemma}
\label{Lem:Exists2-3Move}
Suppose that the labelling $L$ has infinite image.
Suppose that $\calF$ is an $L$--essential foam in $M$.
Suppose that $\calF$ has no cyclic edges.
Then $\calF$ contains an $L$--flippable edge.
\end{lemma}

\begin{proof}
By \reflem{LInfinite5Labels} there is an edge $e$ of $\cover{\calF}$ that is incident to complementary regions with five distinct labels.
By hypothesis, $e$ is not cyclic.
Therefore $\phi_M(e)$ is $L$--flippable.
\end{proof}

\subsection{V-moves}

We use some of the same tools as Matveev, beginning with the V-move~\cite[Definition~1.2.6]{Matveev07}.

\begin{definition}
\label{Def:VMove}
Suppose that $\calF$ is a foam.
Suppose that $v$ is a vertex of $\calF$.
Let $\calN$ be a small regular neighbourhood of $v$.
Suppose that $\delta_+$ and $\delta_-$ are a pair of opposite edges of $\bdy(\calF \cap \calN)$.
Then the \emph{V-move at $v$ along $\delta_+$} is the 0-2 move along $\delta_+$.
(Equivalently, it is also the 0-2 move along $\delta_-$).
\end{definition}


The V-move is shown for both a triangulation and the dual foam in \reffig{VMove}.
Note that there is a reflection symmetry in the resulting triangulation and thus (combinatorially) in the foam.

\begin{figure}[htbp]
\subfloat[]{
\includegraphics[width = 0.43\textwidth]{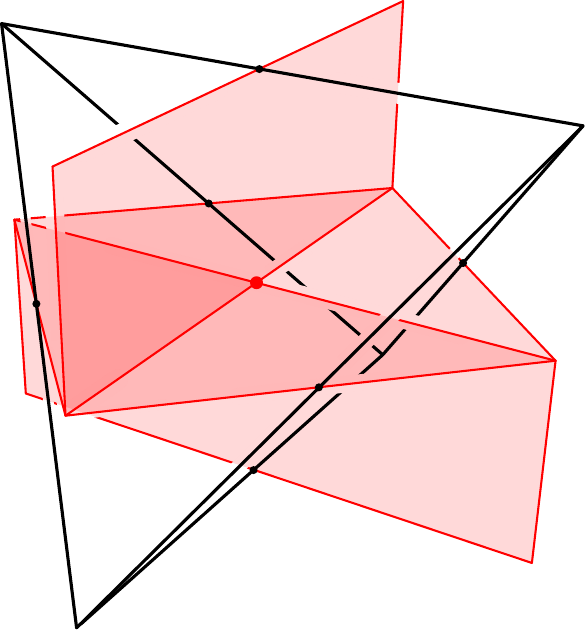}
\label{Fig:VMove1}
}
\quad
\subfloat[]{
\labellist
\hair 2pt \small
\endlabellist
\includegraphics[width = 0.43\textwidth]{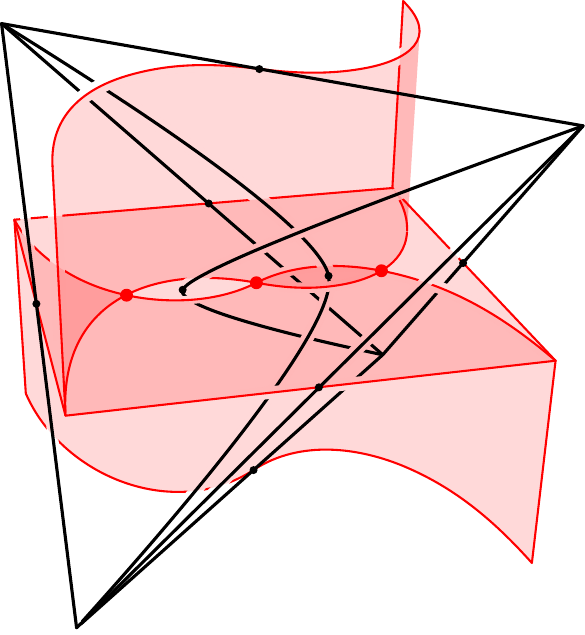}
\label{Fig:VMove2}
}
\caption{The V-move, applied in the right/left direction. Small black dots indicate intersections between the edges of the triangulation and their dual faces in the foam.}
\label{Fig:VMove}
\end{figure}

\begin{lemma}
\label{Lem:VMoveL}
Suppose that $\calF$ is a foam in $M$.
Suppose that applying a V-move to $\calF$ produces $\calF'$.
Then we have the following.
\begin{itemize}
\item $\calF$ is $L$--essential if and only if $\calF'$ is $L$--essential.
\item If $\calF$ has an $L$--flippable edge then so does $\calF'$.  
\item If $e$ is a cyclic edge of $\calF'$ then it has a cyclic ancestor in $\calF$.
\end{itemize}
\end{lemma}

\begin{proof}
The first conclusion follows since no new pair of regions comes into contact as a result of the V-move or its inverse.
For the second conclusion, note that an edge of $\calF$ and its descendant in $\calF'$ (\refdef{Ancestor}) are incident to the same complementary regions.
For the third conclusion, note that
there are four edges in $\calF'$ without ancestors; none of these are cyclic. 
If some other edge in  $\calF'$ is cyclic then its ancestor in $\calF$ is also cyclic.
\end{proof}

\subsection{Avoiding cyclic edges}

As usual, we assume that $M$ is a compact, connected three-manifold with boundary.

\begin{lemma}
\label{Lem:VMoveCyclic}
Suppose that $e$ is a cyclic edge of $\calF$ with both endpoints at $v$.
Then there is an arc $\delta_+$ as in \refdef{VMove} with the additional hypothesis that it meets $e$ in exactly one point.
Moreover, the V-move at $v$ along $\delta_+$ gives a foam $\calF'$ with fewer cyclic edges than $\calF$.
\end{lemma}

\begin{proof}
There are six possibilities for $\delta_+$ in \refdef{VMove}.
One of these meets $e$ in two points and one of these meets $e$ in zero points.
We choose one of the four remaining arcs for $\delta_+$.
The cyclic edge $e$ is then destroyed by the vertices added by the V-move.
See \reffig{VMove}.
By \reflem{VMoveL}, no new cyclic edges are created.
\end{proof}

\begin{lemma}
\label{Lem:RemoveCyclic}
Suppose that $\calF$ is an $L$--essential foam in $M$.
Suppose that $\calF$ contains an $L$--flippable edge.
Then there is a sequence of $L$--essential foams $\calF = \calF_0, \ldots, \calF_n$ where 
\begin{itemize}
\item each foam is related to the next by a 0-2 move, 
\item each foam contains an $L$--flippable edge, and
\item $\calF_n$ has no cyclic edges.
\end{itemize}
\end{lemma}

\begin{proof}
We repeatedly apply \reflem{VMoveCyclic}.
By \reflem{VMoveL}, each resulting foam is $L$--essential and has an $L$--flippable edge.
\end{proof}

The following is a refinement of \refthm{ConnectivityWith0-2}.

\begin{lemma}
\label{Lem:ConnectivityNoCyclic}
Suppose that the labelling $L$ has infinite image.
Suppose that $\calF$ and $\calF'$ are $L$--essential foams in $M$, each with no cyclic edges.
Then there is a path $\calF = \calF_0, \ldots, \calF_n = \calF'$ of $L$--essential foams without cyclic edges where each foam is related to the next by a 2-3, 3-2, 0-2, or 2-0 move.
\end{lemma}

\begin{proof}
Let $\calF = \calG_0, \ldots, \calG_n = \calF'$ be the path of $L$--essential foams (dual to triangulations) given to us by \refthm{ConnectivityWith0-2}.
To be precise in our counting, we say that a move destroys a cyclic edge when the move alters all neighbourhoods of the edge.
A move creates a cyclic edge when the reverse move destroys it.
Note that a single move can destroy one cyclic edge and create another.
Some finite number of cyclic edges are each created and then destroyed in the sequence, since $\calF$ and $\calF'$ have no cyclic edges.
We modify this sequence recursively, so that after each modification one fewer cyclic edge is created.
After a finite number of these modifications, we have the desired sequence.

Suppose that a cyclic edge $e$, with both ends at a vertex $v$, is created by the move $m$ transforming $\calG_p$ into $\calG_{p+1}$.
There are four cases to consider, as the move $m$ is a 2-3, 3-2, 0-2, or 2-0 move.
In each case, we apply a 0-2 move along an arc $\delta$ before the move $m$ to avoid making the cyclic edge.
In some cases, we then apply a second 0-2 move along an arc $\delta'$ followed by a 2-0 move to undo the first 0-2 move.
The foam resulting from this process also results from performing $m$ and then destroying the cyclic edge by performing a V-move at $v$. 
(We choose $\delta$ and $\delta'$ so that all intermediate foams remain $L$--essential.)
As the V-move does not alter the foam outside of a small neighbourhood of $v$, it does not affect later moves in our sequence, until we reach a move that destroys the cyclic edge $e$.
Suppose that this occurs between $\calG_{q-1}$ and $\calG_{q}$.
Viewing the sequence in reverse, between $\calG_{q}$ and $\calG_{q-1}$ we create the cyclic edge.
We can therefore get from $\calG_{q-1}$ to $\calG_{q}$ using the reverse of one of the same four constructions we use for the forward direction.

\begin{figure}[htbp]
\subfloat[Before 2-3 move.]{
\labellist
\hair 2pt \small
\pinlabel $e$ [l] at 294 340
\pinlabel $e$ [l] at 277 80
\pinlabel $\delta$ [b] at 180 27
\endlabellist
\includegraphics[height = 4.5cm]{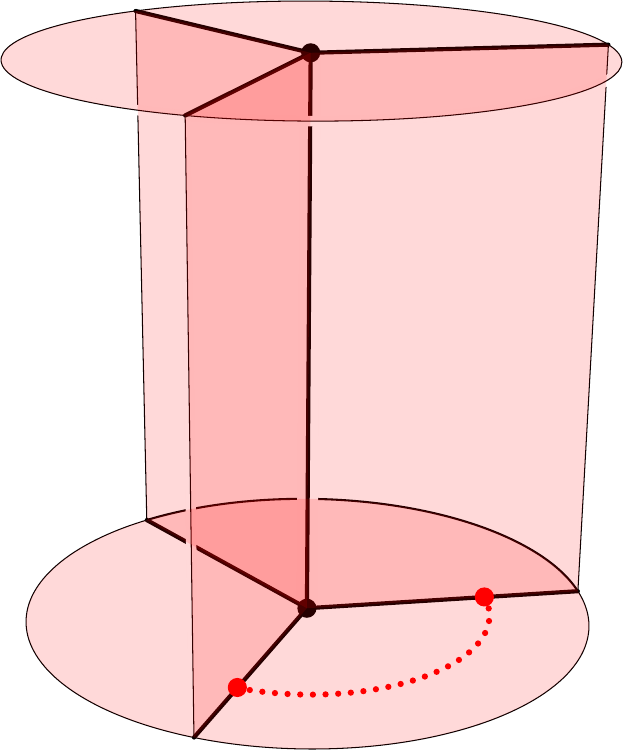}
\label{Fig:2-3BeforeDestroyCyclic}
}
\qquad
\subfloat[After 2-3 move.]{ 
\labellist
\hair 2pt \small
\pinlabel $e$ [l] at 294 340
\pinlabel $e$ [l] at 277 80
\pinlabel $v$ [l] at 190 200
\pinlabel $\delta$ [b] at 180 27
\pinlabel $\delta'$ [tr] at 175 157
\endlabellist
\includegraphics[height = 4.5cm]{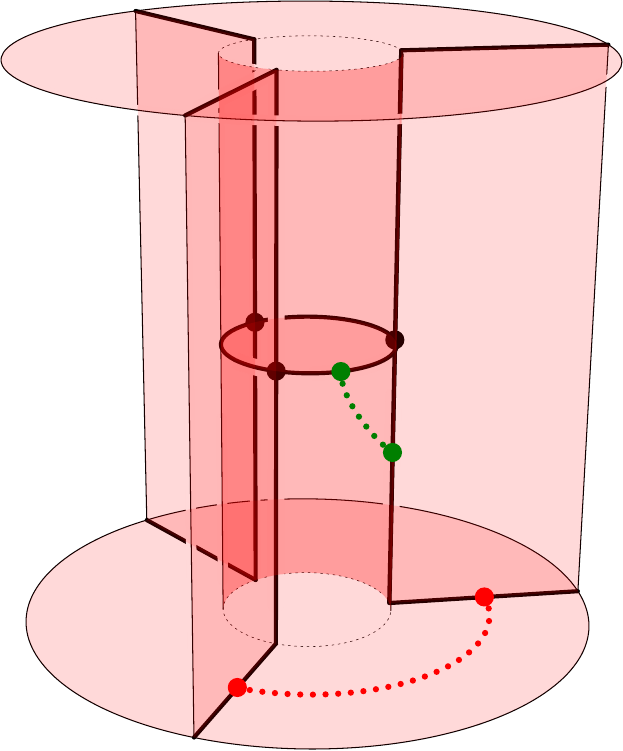}
\label{Fig:2-3AfterDestroyCyclic}
}

\subfloat[Before 3-2 move.]{
\labellist
\hair 2pt \small
\pinlabel $e$ [br] at 70 108
\pinlabel $e$ [l] at 277 80
\pinlabel $\delta$ [b] at 180 27
\endlabellist
\includegraphics[height = 4.5cm]{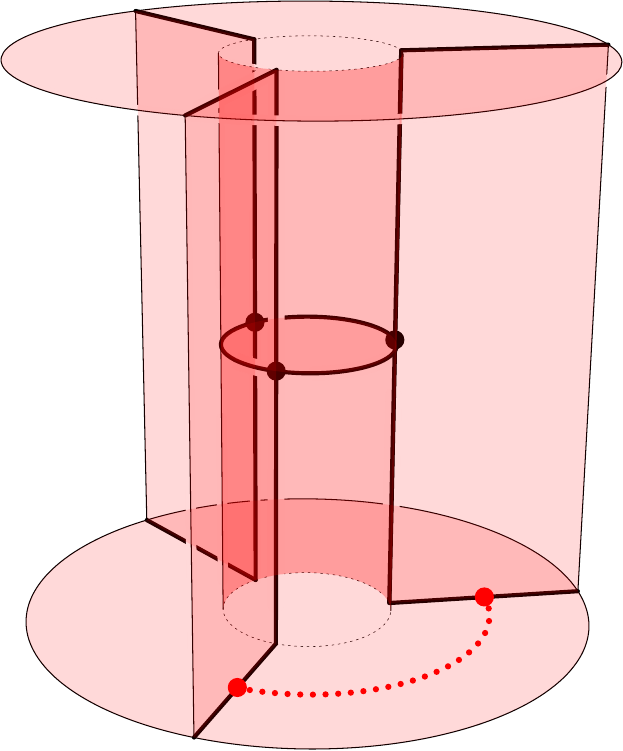}
\label{Fig:3-2BeforeDestroyCyclic}
}
\qquad
\subfloat[After 3-2 move.]{
\labellist
\hair 2pt \small
\pinlabel $e$ [br] at 70 108
\pinlabel $e$ [l] at 277 80
\pinlabel $v$ [bl] at 150 72
\pinlabel $\delta$ [b] at 180 27
\endlabellist
\includegraphics[height = 4.5cm]{Figures/2-3_before_destroy_cyclic}
\label{Fig:3-2AfterDestroyCyclic}
}

\subfloat[Before 0-2 move.]{
\labellist
\hair 2pt \small
\pinlabel $e$ [l] at 455 278
\pinlabel $e$ [l] at 440 45
\pinlabel $w$ [r] at 100 75
\pinlabel $e'$ [t] at 55 25
\pinlabel $\delta$ [b] at 208 30
\endlabellist
\includegraphics[height = 3.7cm]{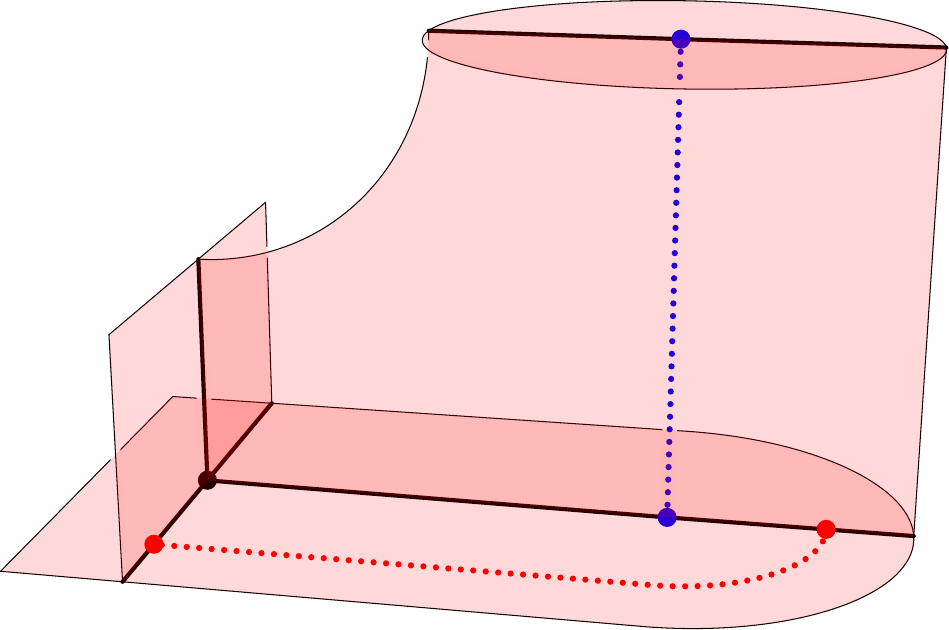}
\label{Fig:0-2BeforeDestroyCyclic}
}
\qquad
\subfloat[After 0-2 move.]{
\labellist
\hair 2pt \small
\pinlabel $e$ [l] at 455 278
\pinlabel $e$ [l] at 440 45
\pinlabel $v$ [l] at 362 163
\pinlabel $w$ [r] at 100 75
\pinlabel $e'$ [t] at 55 25
\pinlabel $\delta$ [b] at 208 30
\pinlabel $\delta'$ [tr] at 330 135
\endlabellist
\includegraphics[height = 3.7cm]{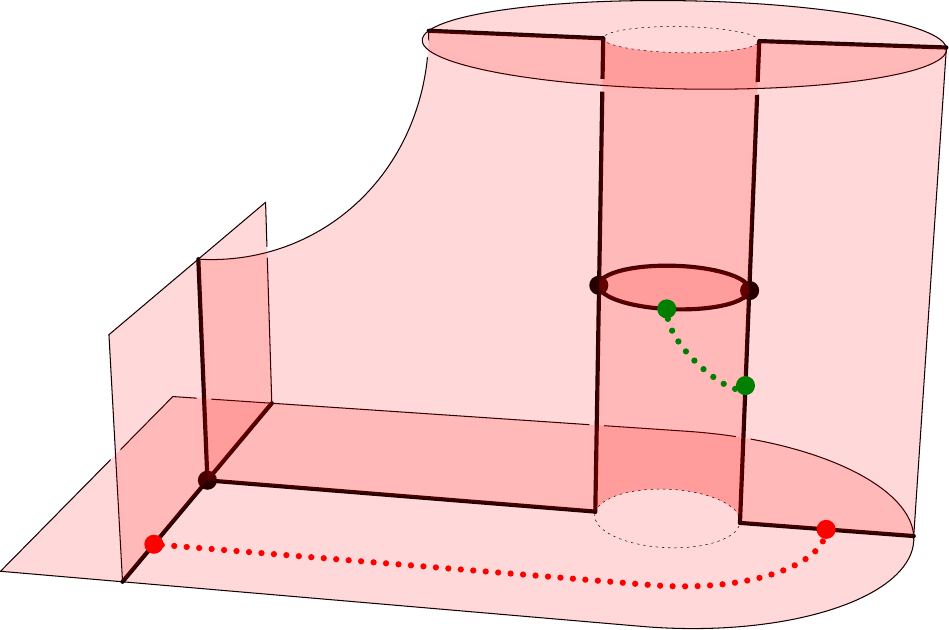}
\label{Fig:0-2AfterDestroyCyclic}
}
\caption{Steps to avoid creating a cyclic edge.}
\label{Fig:DestroyCyclic}
\end{figure}

Suppose that $m$ is a 2-3 move. 
See \reffig{2-3BeforeDestroyCyclic}. 
Breaking symmetry, in $\calG_{p+1}$ the two edge-ends labelled $e$ are connected together at the vertex $v$ after the 2-3 move, as shown in \reffig{2-3AfterDestroyCyclic}.
Before $m$ we apply a 0-2 move along the arc marked $\delta$ in \reffig{2-3BeforeDestroyCyclic}.
This move does not itself introduce another cyclic edge.
Next we apply $m$.
Note that $m$ no longer creates a cyclic edge because of the two extra vertices formed by the 0-2 move along $\delta$.
Next, we apply a 0-2 move along the arc marked $\delta'$ in \reffig{2-3AfterDestroyCyclic}.
Last, we undo the 0-2 move that occurred along $\delta$.

Now suppose that $m$ is a 3-2 move.
See \reffig{3-2BeforeDestroyCyclic}. 
Breaking symmetry, in $\calG_{p+1}$ the two edge-ends labelled $e$ are connected together at the vertex $v$ after the 3-2 move, as shown in \reffig{3-2AfterDestroyCyclic}.
Before $m$ we apply a 0-2 move along the arc marked $\delta$ in \reffig{3-2BeforeDestroyCyclic}.
Again, this move does not itself introduce another cyclic edge.
We then apply $m$.
In this case we do not need to do a second 0-2 move because we already have the result of a V-move at $v$.

Now suppose that $m$ is a 0-2 move.
See \reffig{0-2BeforeDestroyCyclic}. 
Breaking symmetry, in $\calG_{p+1}$ the two edge-ends labelled $e$ are connected together at the vertex $v$ after the 0-2 move, as shown in \reffig{0-2AfterDestroyCyclic}.
Following the edge in the opposite direction, we find a vertex $w$. 
Breaking symmetry again, we find an edge $e'$ incident to $w$ which is not $e$.
Before $m$ we apply a 0-2 move along the arc marked $\delta$ in \reffig{0-2BeforeDestroyCyclic}.
We then apply $m$.
Next, we apply a 0-2 move along the arc marked $\delta'$ in \reffig{0-2AfterDestroyCyclic}.
Last, we undo the 0-2 move that occurred along $\delta$.

Finally, suppose that $m$ is a 2-0 move.
(There are multiple ways in which such a 2-0 move can create a cyclic edge depending on how the four edge ends exiting \reffig{0-2After} are connected to each other.)
In this case, the vertex $v$ exists before $m$ is applied.
Before applying $m$, we apply a 0-2 move to implement a V-move at $v$, chosen as in \reflem{VMoveCyclic}.
Again we do not need to do a second 0-2 move because we already have the result of a V-move at $v$.
\end{proof}

The following is the final tool we use to prove \refthm{Connectivity}.

\begin{proposition}
\label{Prop:Make0-2}
Suppose that the labelling $L$ has infinite image.
Suppose that $\calF$ is an $L$--essential foam in $M$.
Suppose that applying a 0-2 move along an arc $\delta$ in a face of $\calF$ produces $\calF[\delta]$, which is also $L$--essential.
Suppose that each of $\calF$ and $\calF[\delta]$ has an $L$-flippable edge.
Then there is a path $\calF = \calF_0, \ldots, \calF_n = \calF[\delta]$ of $L$--essential foams where each foam is related to the next by a 2-3 or 3-2 move.
\end{proposition}

\begin{proof}[Proof of \refthm{Connectivity}]
Let $\calF$ and $\calF'$ be foams dual to the triangulations $\calT$ and $\calT'$ given in the statement of \refthm{Connectivity}.
We apply \reflem{RemoveCyclic} to $\calF$ and to $\calF'$ to produce sequences of $L$--essential foams ending at $\calG$ and $\calG'$ say (respectively).
Thus $\calG$ and $\calG'$ have no cyclic edges.
\reflem{RemoveCyclic} also tells us that every foam in these sequences has an $L$--flippable edge.

We apply \reflem{ConnectivityNoCyclic} to produce a sequence of $L$--essential foams without cyclic edges connecting $\calG$ to $\calG'$.
By \reflem{Exists2-3Move}, every foam in the sequence connecting $\calG$ to $\calG'$ has an $L$--flippable edge.

Concatenating the three sequences together, we obtain a sequence of $L$--essential foams connecting $\calF$ to $\calF'$.
Consecutive foams are related by a 2-3, 3-2, 0-2, or 2-0 move, and each foam contains an $L$--flippable edge.
We then use \refprop{Make0-2} to replace each 0-2 or 2-0 move in the sequence with a sequence of 2-3 and 3-2 moves.
\end{proof}


The proof of \refprop{Make0-2} is quite difficult;
it begins in \refsec{FollowingMatveev} and takes up the remainder of the paper.

\subsection{Connecting the isolated}

We can now connect isolated triangulations to $\TT^\circ(M,L)$ using V-moves.

\begin{lemma}
\label{Lem:IsolatedVMove}
Suppose that $L \from \Delta_M \to \calL$ has infinite image.
Suppose that $\calF$ is an $L$--essential foam in $M$. 
Suppose that $\calF$ contains no $L$--flippable edges.
Then there is a V-move on $\calF$ that results in an $L$--essential foam $\calF'$ which contains an $L$--flippable edge.
\end{lemma}

\begin{proof}
By \reflem{LInfinite5Labels}, there is an edge $e$ of $\cover{\calF}$ that is incident to complementary regions with five distinct labels. 
Since $\calF$ contains no $L$--flippable edges, $\phi_M(e)$ must be cyclic.
Let $\calN$ be a small regular neighbourhood of $v$, the unique endpoint of $\phi_M(e)$.
Let $\delta_+$ be as given in \reflem{VMoveCyclic}.
We obtain $\calF'$ by performing a 0-2 move along $\delta_+$.
By \reflem{VMoveL}, the foam $\calF'$ is $L$--essential. 
Also, the descendant of $\phi_M(e)$ not contained in $\calN$ is $L$--flippable.
\end{proof}

\begin{remark}
\label{Rem:MoreIsolated}
Suppose that, in the proof of \reflem{IsolatedVMove}, we instead choose $\delta_+$ to meet $e$ two or zero times.
Then the resulting foam $\calF'$ again has no $L$--flippable edge.
Thus if there is one isolated $L$--essential foam, there are infinitely many.
\end{remark}

\section{Following Matveev}
\label{Sec:FollowingMatveev}

Without the $L$--essentiality condition, \refprop{Make0-2} is proved by Matveev (combining Lemma~1.2.11 and Proposition~1.2.8 of~\cite{Matveev07}).

We will give a series of constructions that prove \refprop{Make0-2} under increasingly general circumstances.
Following \cite[Proposition~1.2.8]{Matveev07}, we begin with the construction of a V-move in \reflem{ImplementVMove}, where the arc $\delta$ that the 0-2 move is to be applied along cuts off a single vertex $v$ of its model face.

\reflem{ImplementVMove} requires that some edge incident to the vertex $v$ is $L$--flippable.
In the completely general case there may be no $L$--flippable edge anywhere near $\delta$, so the general construction must start work at such an edge and work towards $\delta$.
Thus we require a ``pre-processing'' stage (see Sections~\ref{Sec:LooseningMove} and~\ref{Sec:BringDistant}) where we generate good circumstances around $\delta$ that allows us to perform the 0-2 move.
\reflem{Do0-2ManyEdges} gives an implementation of a 0-2 move under these good circumstances.
That lemma relies on \reflem{Do0-2ThreeSides}, which implements a 0-2 move under even more stringent circumstances.

\begin{lemma}
\label{Lem:ImplementVMove}
Suppose that applying a V-move to a vertex $v$ of an $L$--essential foam $\calF$ produces $\calF'$.
Suppose that some edge $e$ incident to $v$ is $L$--flippable.
Then there is a sequence of 2-3 and 3-2 moves from $\calF$ to $\calF'$ such that all foams are $L$--essential.
\end{lemma}

\begin{proof}
\reffig{VMoveConstruction} shows how to implement a V-move using three 2-3 moves followed by a 3-2 move.
(Note that if we rotate our picture of the foam around $e$ before applying these moves we also obtain the other two V-moves on $v$.)
By hypothesis, the first move along $e$ is possible and does not introduce an $L$--inessential face.
Each of the subsequent moves takes place on a collection of distinct vertices because those vertices were produced by the previous moves.
Thus each move can be applied.
Moreover, at each step, faces without ancestors are only created between regions that already have a face in common.
Thus the moves do not produce an $L$--inessential face.
\end{proof}

\begin{figure}[htbp]
\labellist
\hair 2pt \small
\pinlabel $e$ [b] at 60 196
\endlabellist
\includegraphics[width = \textwidth]{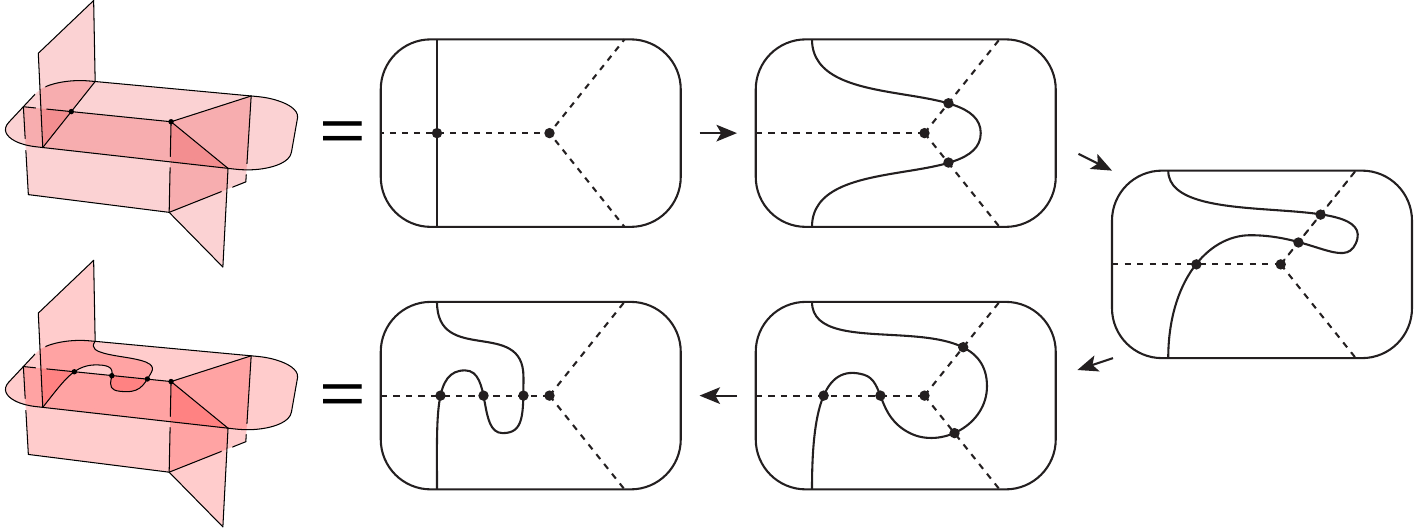}
\caption{A V-move can be implemented by three 2-3 moves followed by a 3-2 move. Adapted from \cite[Figure~1.15]{Matveev07}.}
\label{Fig:VMoveConstruction}
\end{figure}

It will be convenient for drawing pictures and describing the combinatorics to talk about 0-2 moves in the following way. 

\begin{definition}
\label{Def:Snakelet}
Suppose that $\calF$ is a foam in a manifold $M$.
Suppose that $f$ is a co-oriented model face of $\calF$.
Suppose that $\delta$ is an oriented arc properly embedded in $f$ which is disjoint from the vertices of $f$.
(It will be convenient to conflate $\delta$ with its image in $M$.)
Let $\calN$ be a regular neighbourhood of $\delta$ in $M$.
Let $S$ be the component of $\calN - \calF$ meeting the interior of $\delta$ and pointed at by the co-orientation of $\calF$.
We say that $S$ is the \emph{snakelet} generated by $\delta$.

After taking the closure, the boundary of $S$ consists of two bigons 
$b$ and $b'$ and two rectangles $r$ and $r'$.
Suppose that $b$ is the bigon pointed at by $\delta$ and suppose that $r$ is the rectangle contained in $f$.
We call $b$ and $b'$ the \emph{head} and \emph{tail} of $S$ respectively.
We call $r$ and $r'$ the \emph{belly} and \emph{back} of $S$ respectively.
See \reffig{Snakelet}.
\end{definition}

\begin{figure}[htbp]
\subfloat[]{
\labellist
\hair 2pt \small
\pinlabel $\delta$ [r] at 134 200
\endlabellist
\includegraphics[height = 4cm]{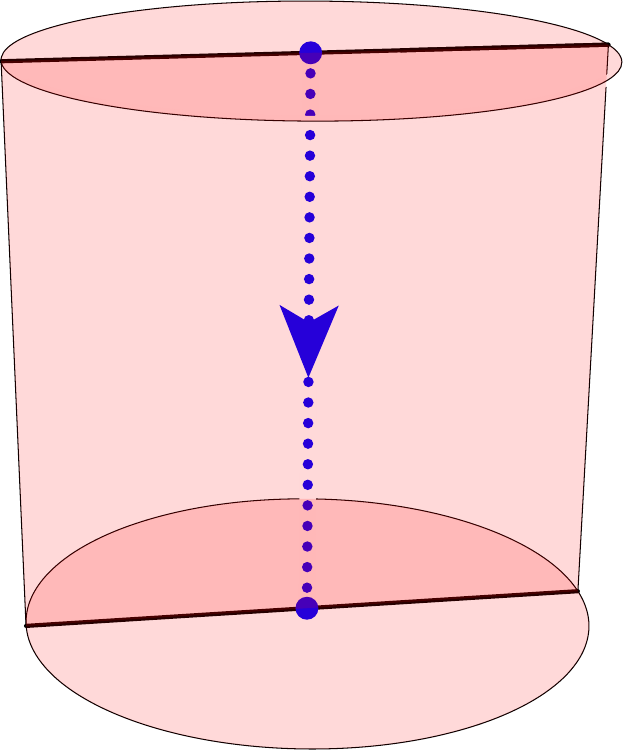}
\label{Fig:SnakeletBefore}
}
\qquad
\subfloat[]{
\includegraphics[height = 4cm]{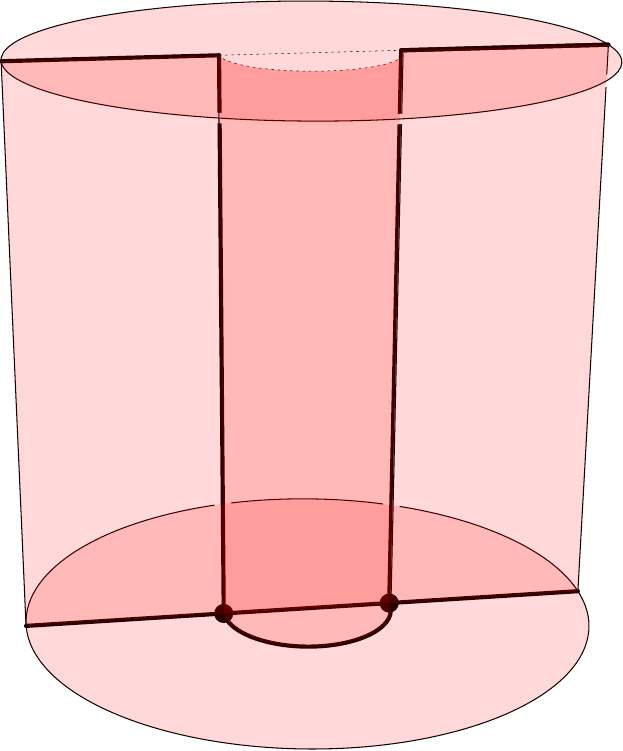}
\label{Fig:SnakeletAfter}
}
\caption{A 0-2 move generated by a snakelet along the arc $\delta$. Since $\delta$ points down, the head of the snakelet is at the bottom of \reffig{SnakeletAfter} while the tail is at the top. Compare with \reffig{0-2}.}
\label{Fig:Snakelet}
\end{figure}

Note that the foam $\calF' = (\calF \cup r') - \interior(b')$ is combinatorially identical to the foam obtained by performing a 0-2 move along $\delta$.

\begin{figure}[htbp]
\subfloat[]{
\includegraphics[width = 0.31\textwidth]{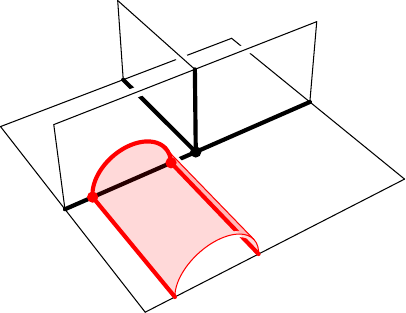}
\label{Fig:MoveSnakeletHeadPastVertex0}
}
\subfloat[]{
\includegraphics[width = 0.31\textwidth]{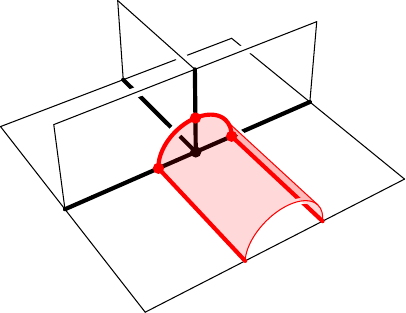}
\label{Fig:MoveSnakeletHeadPastVertex1}
}
\subfloat[]{
\includegraphics[width = 0.31\textwidth]{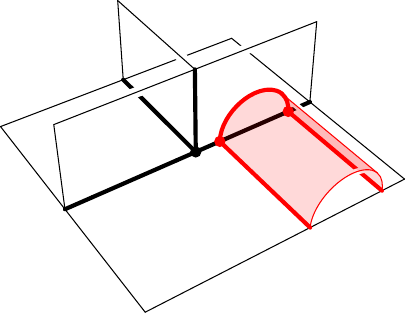}
\label{Fig:MoveSnakeletHeadPastVertex2}
}
\caption{Moving the head of a snakelet past a vertex.}
\label{Fig:MoveSnakeletHeadPastVertex}
\end{figure}

We now discuss how to move the head of a snakelet across a vertex.

\begin{lemma}
\label{Lem:SnakeletVertex}
Suppose that $\delta$ and $\delta'$ are arcs in $f$ that have the same start point but differ by an isotopy moving the terminal point from one edge of $f$ to the next.
Suppose that $\calF[\delta]$ and $\calF[\delta']$ are $L$--essential.
Then we can connect $\calF[\delta]$ to $\calF[\delta']$ by a 2-3 move followed by a 3-2 move without introducing an $L$--inessential face.
\end{lemma}

\begin{proof}
See \reffig{MoveSnakeletHeadPastVertex}.
Since the arcs move by isotopy, their terminal points do not cross their initial points.
Note that the only new regions that come into contact are also in contact in $\calF[\delta']$.
Therefore all faces are $L$--essential throughout.
Also note that the edge along which we apply the 2-3 move cannot be cyclic since its endpoints are distinct: 
one endpoint is a vertex of $f$ and the other is part of the snakelet.
\end{proof}

\subsection{Moving snakelets to realise some 0-2 moves} 
\label{Sec:MovingSnakelets}

Throughout this subsection we make the following assumptions.

\begin{hypotheses}
\label{Hyp:MovingSnakelets}
Suppose that $M$ is a compact, connected three-manifold with boundary.
Suppose that $L$ is a labelling of $\Delta_M$. 
Suppose that $\calF$ is an $L$--essential foam in $M$.
Suppose that $f$ is a model face of $\cover{\calF}$.
Suppose that $\delta$ is an arc properly embedded in $f$ and disjoint from the model vertices of $f$.
Suppose that $\calF$ and $\calF[\phi_M(\delta)]$ are $L$--essential.
\end{hypotheses}

To simplify our notation, for the remainder of the paper we will conflate $\delta$ with $\phi_M(\delta)$ and $f$ with $\phi_M(f)$.

\begin{definition}
\label{Def:Sides}
The two components of $f - \delta$ are the \emph{sides} of $\delta$.
\end{definition}

To aid our exposition, we choose a co-orientation on $f$ and an orientation for $\delta$.
This allows us to realise the 0-2 move along $\phidelta$ as generating a snakelet $S$, as in \refdef{Snakelet}.

\begin{definition}
\label{Def:FaceFromSide}
Suppose that $s$ is a side of $\delta$.
Recall that $r$ is the belly of the snakelet: a small regular neighbourhood, in $f$, of $\delta$.
Choose a point $x$ in $s - r$.
We denote by $f_s$ the descendant (under the 0-2 move) of $f$ in $\calF[\delta]$ containing $x$.
\end{definition}

\begin{remark}
\label{Rem:FaceFromSide}
If $f$ has no self-gluings then $f_s$ is a subset of $f$.
However, when a model edge of $s \cap \bdy f$ other than the first or last meets $\delta$ then $f_s$ is more complicated.
To see this, consider \reffig{SnakeletBefore}.
In that figure, the dotted line indicating $\delta$ is on the face $f$. 
However, suppose that one of the other two disks at the top of the figure lies in (a translate of) $s$.
In this case $f_s$ extends along either the belly or the back of the snakelet to meet the head of the snakelet.
If instead one of the two disks at the bottom of the figure lies in (a translate of) $s$ then $f_s$ meets the head of the snakelet directly.
\end{remark}

Suppose that $s$ is one of the sides of $\delta$.
We orient $s \cap \bdy f$ from the initial point to the terminal point of $\delta$.
We name the model edges of $f$, that meet $s \cap \bdy f$ as $e_0, \ldots e_N$, with index increasing in the direction of the orientation.
Note that $s$ meets at least two model edges;
if it did not then the foam $\calF[\delta]$ would have an $L$--inessential face.
We also name the model vertex of $f$ where $e_i$ meets $e_{i+1}$ as $v_i$.

\begin{definition}
Suppose that $e$ is a model edge of $f$.
Let $x$ be a point in the interior of $e$;
let $\eta(x)$ be a small regular neighbourhood of $x$, taken in $\cover{M}$.
Note that $\eta(x)$ is a three-ball.
Let $y$ be a point in the intersection of $\eta(x)$ with an open collar of the model edge $e$ taken in $f$.
Let $\eta'$ be the component of $\eta(x) - \cover{\calF}$ whose closure does not contain $y$.
The \emph{outer region} for $e$ is the component $E$ of $\cover{M} - \cover{\calF}$ that contains $\eta'$.
We also say that $E$ is an \emph{outer region for} $f$.
\end{definition}

Let $E_i$ be the outer region for $e_i$.
Recall that $\phi \from \cover{M} \to M$ is the universal covering map.

\begin{lemma}
\label{Lem:Do0-2ThreeSides}
Assuming \refhyp{MovingSnakelets}, suppose that $s$ is a side of the arc $\delta$.
Suppose that $s$  meets only three model edges, $e_0$, $e_1$, and $e_2$.
Let $E$ = $E_1$.
Suppose that $L(E)$ appears precisely once as a label of an outer region for an edge of $s$, and precisely once as a label for an outer region for an edge of $f_s$.
Suppose that $e_1$ is not cyclic.
Suppose that $\phi_M(e_1)$ is distinct from $\phi_M(e_0)$ and $\phi_M(e_2)$.
Then there is a path $\calF = \calF_0, \ldots, \calF_n = \calF[\delta]$ of $L$--essential foams of $M$ where each foam is related to the next by a 2-3 or 3-2 move.
\end{lemma}

\begin{figure}[htbp]
\subfloat[]{
\labellist
\hair 2pt \small
\pinlabel $E$ at 65 190
\pinlabel $A$ at 110 15
\pinlabel $B$ at 225 220
\pinlabel $e_0$ [tr] at 148 58
\pinlabel $v_0$ [l] at 80 85
\pinlabel $e_1$ [br] at 109 119
\pinlabel $v_1$ [t] at 151 151
\pinlabel $e_2$ [bl] at 209 137
\pinlabel $f$ at 180 100
\endlabellist
\includegraphics[width = 0.48\textwidth]{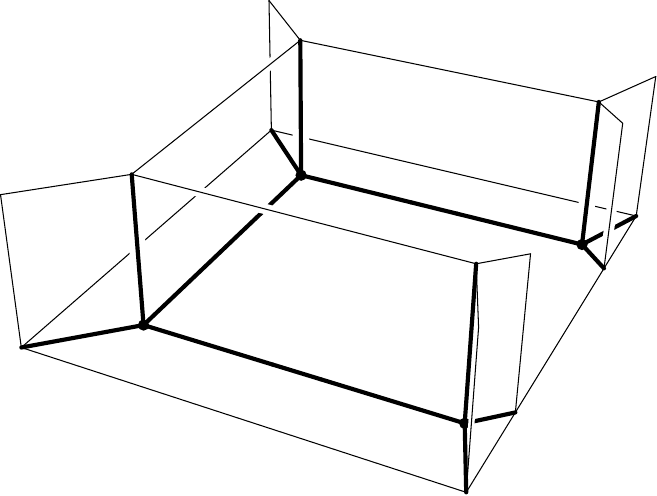}
\label{Fig:ThreeEdges0}
}

\subfloat[]{
\includegraphics[width = 0.41\textwidth]{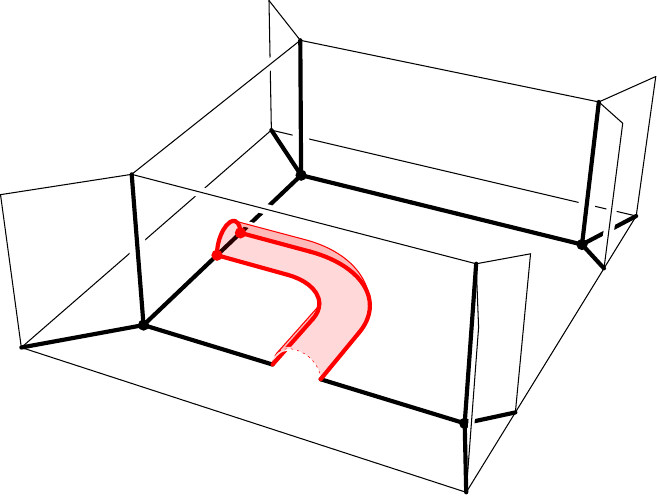}
\label{Fig:ThreeEdges1}
}
\subfloat[]{
\includegraphics[width = 0.41\textwidth]{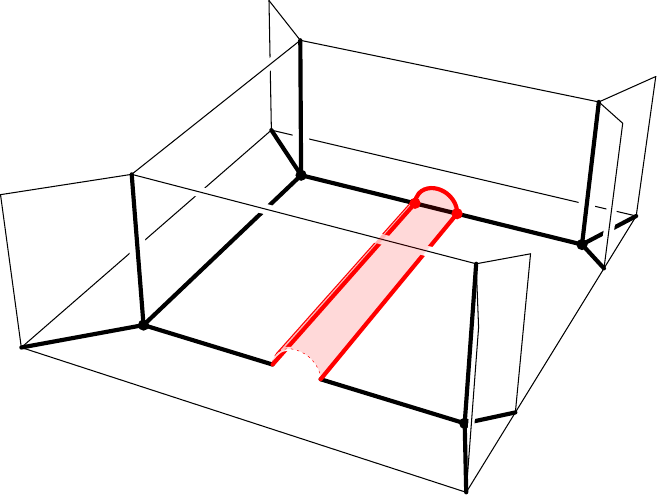}
\label{Fig:ThreeEdges2}
}
\caption{}
\label{Fig:ThreeEdgesEasy}
\end{figure}

\begin{figure}[htbp]
\subfloat[]{
\labellist
\hair 2pt \tiny
\pinlabel $e'_2$ [b] at 119 91
\pinlabel $e'_1$ [r] at 77 96
\pinlabel $v_1$ [t] at 95 97
\endlabellist
\includegraphics[width = 0.31\textwidth]{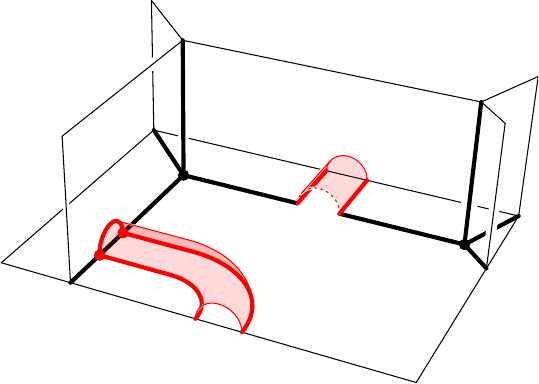}
\label{Fig:ThreeEdges1b}
}
\subfloat[]{
\includegraphics[width = 0.31\textwidth]{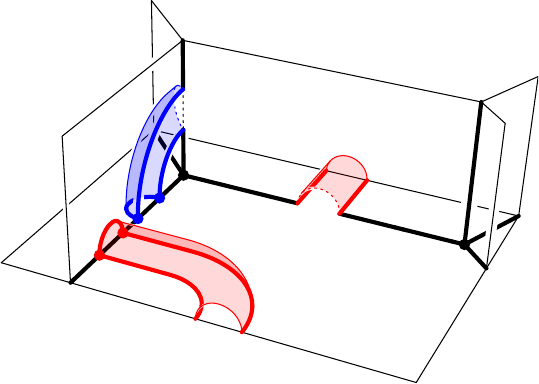}
\label{Fig:ThreeEdges2b}
}
\subfloat[]{
\includegraphics[width = 0.31\textwidth]{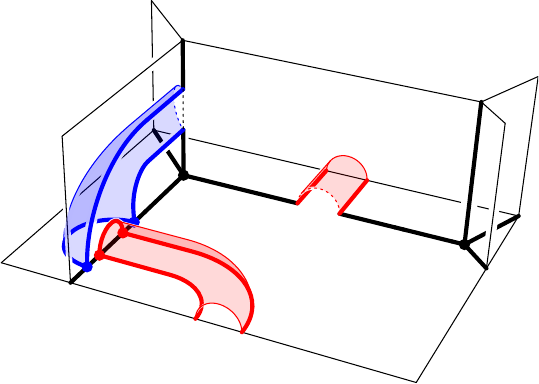}
\label{Fig:ThreeEdges3b}
}

\subfloat[]{
\includegraphics[width = 0.31\textwidth]{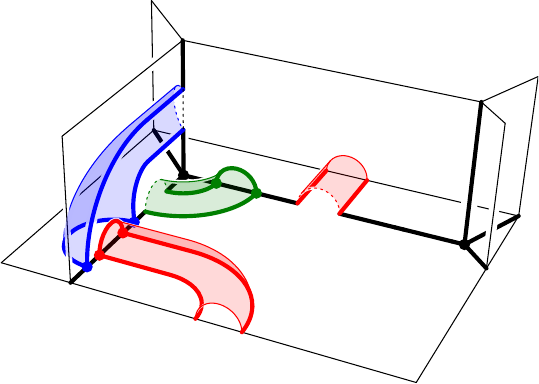}
\label{Fig:ThreeEdges4b}
}
\subfloat[]{
\includegraphics[width = 0.31\textwidth]{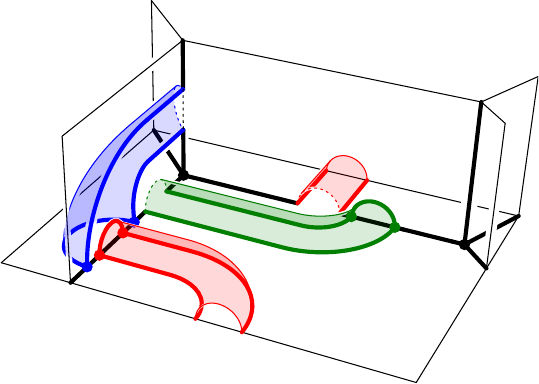}
\label{Fig:ThreeEdges5b}
}
\subfloat[]{
\includegraphics[width = 0.31\textwidth]{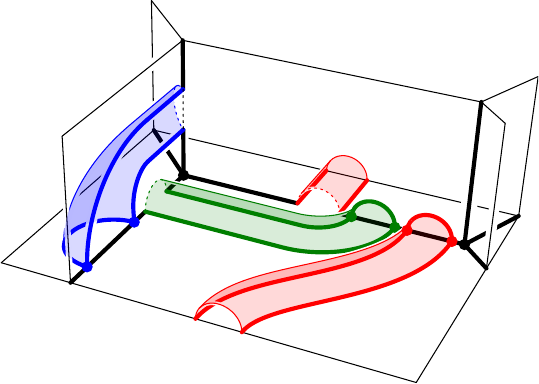}
\label{Fig:ThreeEdges6b}
}

\subfloat[]{
\includegraphics[width = 0.31\textwidth]{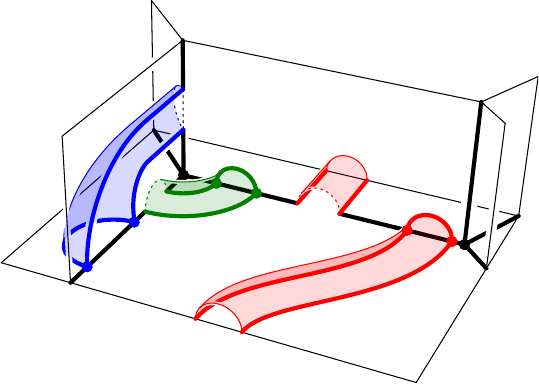}
\label{Fig:ThreeEdges7b}
}
\subfloat[]{
\includegraphics[width = 0.31\textwidth]{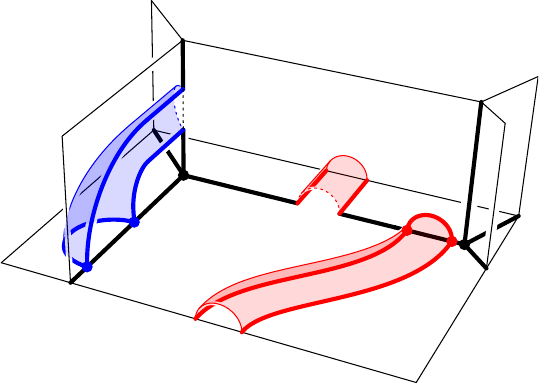}
\label{Fig:ThreeEdges8b}
}
\subfloat[]{
\includegraphics[width = 0.31\textwidth]{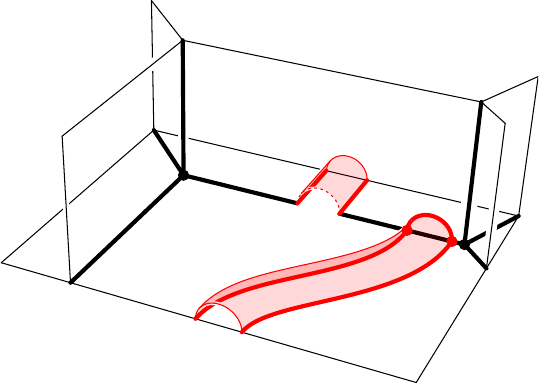}
\label{Fig:ThreeEdges9b}
}
\caption{}
\label{Fig:ThreeEdgesSwap}
\end{figure}

\begin{proof}
Let $A = E_0$ and $B = E_2$.
See \reffig{ThreeEdges0}.
Let $N$ and $S$ be the regions of $\cover{M} - \cover{\calF}$ incident to the interior of $f$.
The labels on $A$, $B$, $E$, $N$, and $S$ are all distinct because $\calF[\delta]$ is $L$--essential.
Because $L(A) \neq L(B)$ and $e_1$ is not cyclic, we have that $e_1$ is $L$--flippable.
By \reflem{ImplementVMove} we may perform a V-move at the vertex $v_0$ where $e_0$ meets $e_1$. 
See \reffig{ThreeEdges1}.
This produces the \emph{red snakelet}.
Its head is on $e_1$ and its tail is on $e_0$.
(Here we think of $e_1$ and $e_0$ as subsets of $M$, since after the V-move, they are no longer edges of the foam.)
Since $\phi_M(e_1)$ is distinct from $\phi_M(e_0)$, we may isotope the head of the red snakelet to the other end of $e_1$.
Applying \reflem{SnakeletVertex}, we move the head of the red snakelet onto $e_2$.
At this stage we are done unless we are in the case that $\phi_M(e_0) = \phi_M(e_2)$, and the head and the tail of the red snakelet are in the wrong order along the edge $e_2$.
(Matveev swaps the ends of a snakelet in~\cite[Figure~1.19]{Matveev07}; maintaining $L$--essentiality requires additional work as follows.)

So, suppose that $\phi_M(e_0) = \phi_M(e_2)$.
\reffig{ThreeEdges1b} shows the foam $\calF'$ obtained after performing the V-move in this case. 
(In \reffig{ThreeEdgesSwap} we no longer draw the edge $e_0$ so that we may continue to illustrate all cases without additional figures.
We would see different configurations of snakelets along $e_0$ depending on how the orientations of $e_0$ and $e_2$ relate to the orientation of $\phi_M(e_0) = \phi_M(e_2)$.)
The stabiliser of $e_0$ in $\pi_1(M)$ acting on $\cover{M}$ is trivial.
Thus there is a unique $\gamma \in \pi_1(M)$ taking $e_0$ to $e_2$.

Let $e_1'$ and $e_2'$ be the edges of $\cover{\calF'}$ leaving $v_1$ in the direction of $e_1$ and $e_2$ (in the original foam $\cover{\calF}$).
See \reffig{ThreeEdges1b}.
Note that neither $e_1'$ nor $e_2'$ is cyclic because the vertex at one end, $v_1$, is part of the original foam $\cover\calF$, while the vertex at the other end is one of the two vertices at the head of the red snakelet.
The regions $E$ and $\gamma(E)$ meet the endpoints of $e'_2$ but not the interior of $e'_2$.
There are two cases depending on whether or not $L(E) = L(\gamma(E))$.

\begin{case*}[$L(E) = L(\gamma(E))$]
In this case we perform a V-move at $v_1$ as shown in \reffig{ThreeEdges2b}.
This is possible by \reflem{ImplementVMove} using the fact that $e'_1$ is not cyclic and that the regions at its endpoints, $A$ and $B$, have distinct labels.
This builds the \emph{blue snakelet}.
We then apply \reflem{SnakeletVertex} to move the head of the red snakelet past one vertex of the blue snakelet.
See \reffig{ThreeEdges3b}.
The hypotheses of \reflem{SnakeletVertex} are satisfied because $L(A) \neq L(B)$.

Now the regions at the ends of $e'_2$ have labels $L(E)$ and $L(\gamma(B))$.
(The latter is because $B$ is at the end of $\delta$, and therefore $\gamma(B)$ is at the end of $\gamma(\delta)$.)
Since $\calF$ is $L$--essential, we have that $L(E)$ is not equal to $L(B)$.
By equivariance, $L(\gamma(E))$ is not equal to $L(\gamma(B))$.
By assumption, $L(E) = L(\gamma(E))$.
Thus the labels $L(E)$ and $L(\gamma(B))$ (of the regions at the ends of $e'_2$) are distinct.  
Since $e_2'$ remains non-cyclic it is $L$--flippable.
We apply \reflem{ImplementVMove} to perform a V-move at $v_1$ as shown in \reffig{ThreeEdges4b}.
This builds the \emph{green snakelet}.

Next, we slide the head of the green snakelet along the edge $e'_2$, into the red snakelet, over the two vertices of the red snakelet (using \reflem{SnakeletVertex}), and out of the red snakelet to reach \reffig{ThreeEdges5b}.
Again we are able to do this without creating any $L$--inessential faces because  $L(E) \neq L(\gamma(B))$.

Next, we slide the head of the red snakelet off of the belly of the blue snakelet and along the back of the green snakelet to reach \reffig{ThreeEdges6b}.
In this process, the head of the red snakelet meets regions $B$, $E$, and then $B$ again.
The labels on these are distinct from the label on $A$.
We then slide the green snakelet back through the inside of the red snakelet and back to its original position, as shown in \reffig{ThreeEdges7b}.
In this process, the head of the green snakelet meets regions $B$, $\gamma(B)$, and then $B$ again.
The labels on these are distinct from the label on $E$.

Finally we perform reverse V-moves (\reflem{ImplementVMove}) to remove the green snakelet (as shown in \reffig{ThreeEdges8b}), and the blue snakelet, to reach \reffig{ThreeEdges9b}.
To remove the green snakelet we again use the fact that $L(E) \neq L(\gamma(B))$.
To remove the blue snakelet we use the fact that $L(A) \neq L(B)$.
\end{case*}

\begin{case*}[$L(E) \neq L(\gamma(E))$]
In this case we do not add the blue snakelet and we omit all steps mentioning it.
The various applications of \reflem{ImplementVMove} and \reflem{SnakeletVertex} are justified similarly. 
The one subtle step is when we slide the green snakelet back through the inside of the red snakelet and back to its original position, as shown in \reffig{ThreeEdges7b} (again the blue snakelet is not present).
Here we require that $L(E)$ is distinct from $L(\gamma(B))$.
This holds by the hypothesis that $L(E)$ appears precisely once as the label of an outer region for $f_s$. \qedhere
\end{case*}
\end{proof}

\begin{lemma}
\label{Lem:Do0-2ManyEdges}
Assuming \refhyp{MovingSnakelets}, 
suppose that the side $s$ meets $N \geq 3$ model edges.
Suppose there is an index $k$ with the following properties.
Set $E = E_k$.
Suppose that $L(E)$ appears precisely once as a label of an outer region for an edge of $s$, and precisely once as a label of an outer region for an edge of $f_s$.
Suppose that $e_{k-1}$, $e_{k}$, and $e_{k+1}$ are not cyclic.
Suppose that the image $\phi_M(e_k)$ under the covering map is distinct from $\phi_M(e_i)$ for $i \neq k$.
Then there is a path $\calF = \calF_0, \ldots, \calF_n = \calF[\delta]$ of $L$--essential foams of $M$ where each foam is related to the next by a 2-3 or 3-2 move.
\end{lemma}

\begin{figure}[htbp]
\subfloat[]{
\labellist
\hair 2pt \small
\pinlabel $A$ [r] at 35 153
\pinlabel $E$ [b] at 210 270
\pinlabel $B$ [l] at 445 90
\endlabellist
\includegraphics[width = 0.45\textwidth]{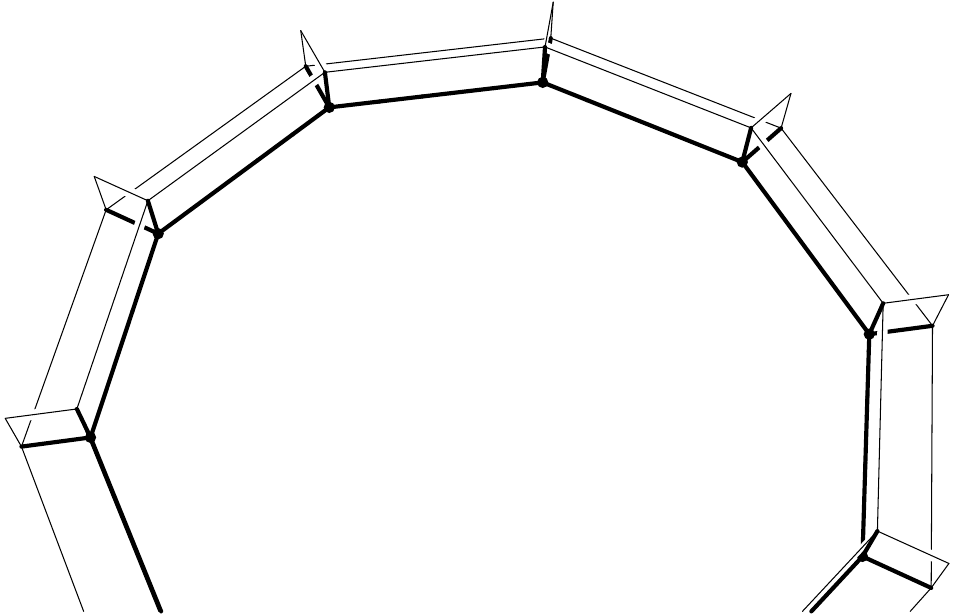}
\label{Fig:ManyEdges0}
}
\quad
\subfloat[]{
\includegraphics[width = 0.45\textwidth]{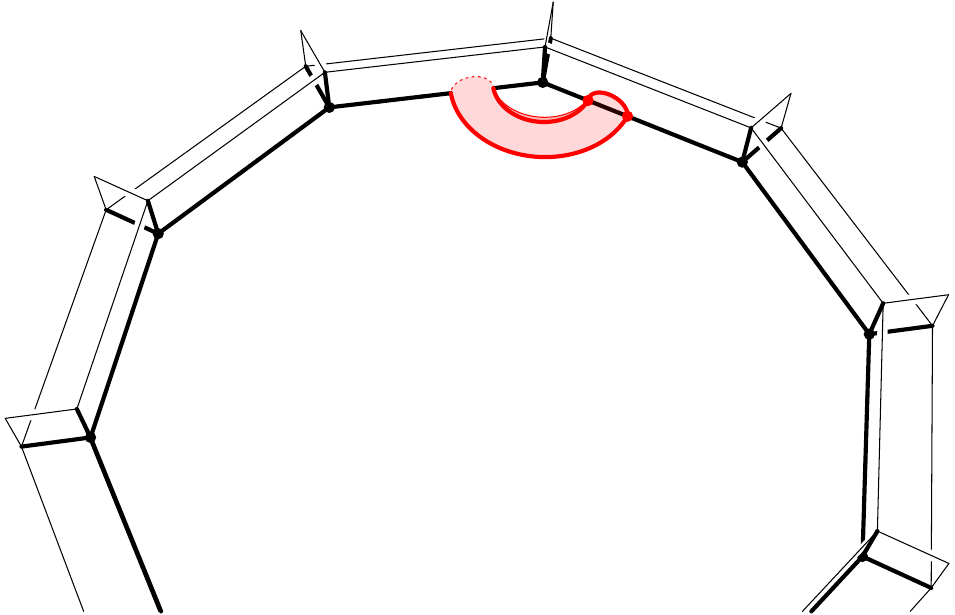}
\label{Fig:ManyEdges1}
}

\subfloat[]{
\includegraphics[width = 0.45\textwidth]{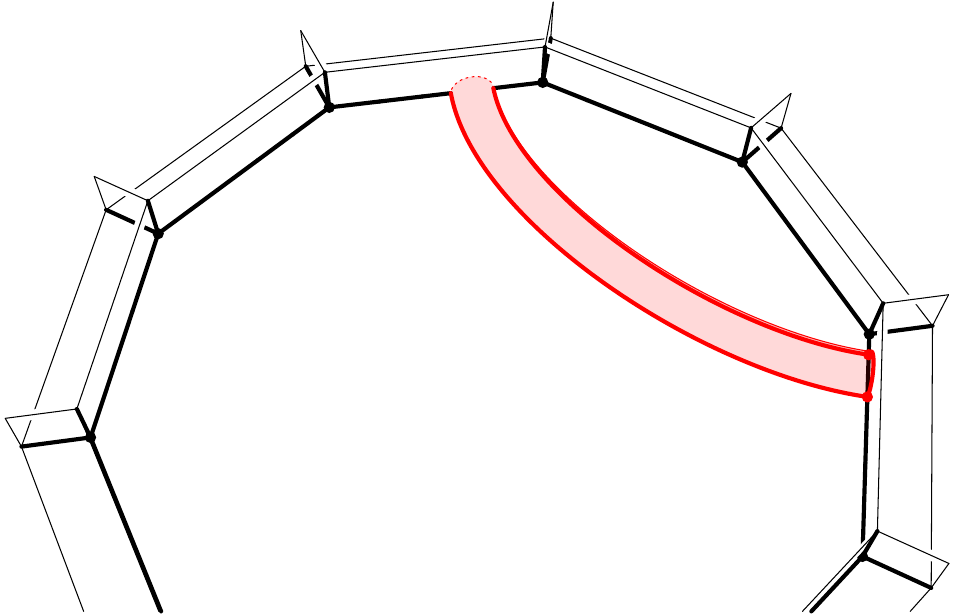}
\label{Fig:ManyEdges2}
}
\quad
\subfloat[]{
\includegraphics[width = 0.45\textwidth]{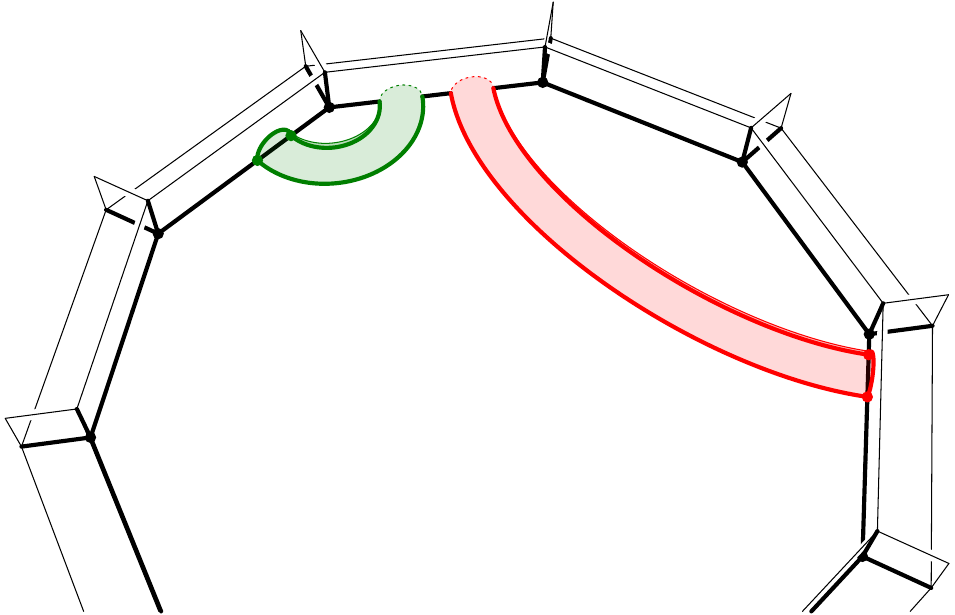}
\label{Fig:ManyEdges3}
}

\subfloat[]{
\includegraphics[width = 0.45\textwidth]{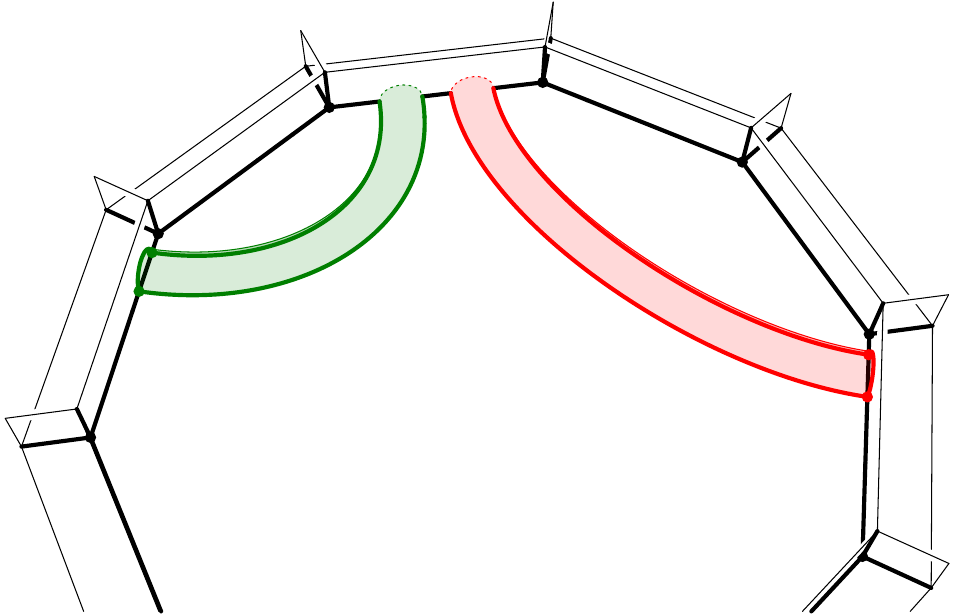}
\label{Fig:ManyEdges4}
}
\quad
\subfloat[]{
\includegraphics[width = 0.45\textwidth]{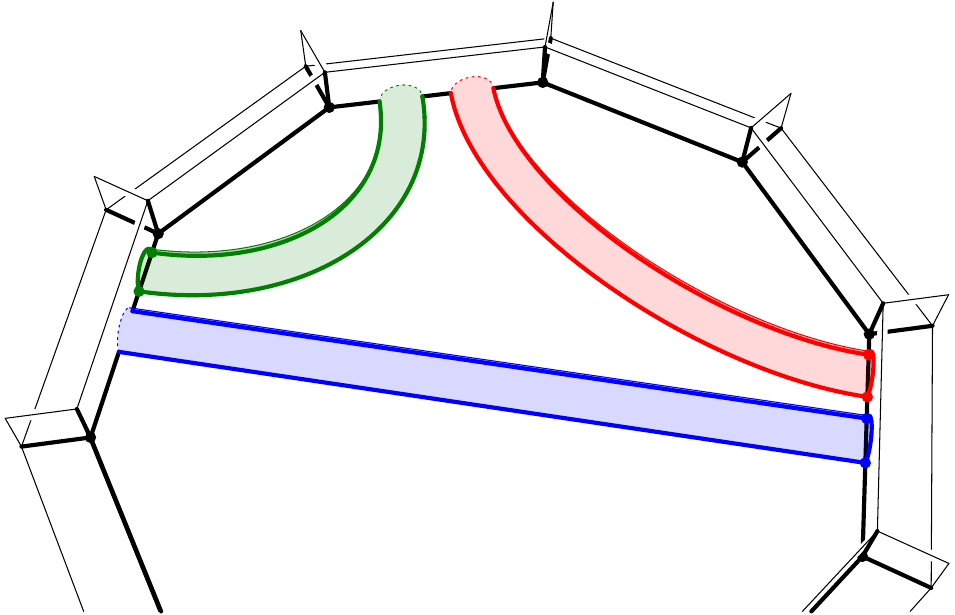}
\label{Fig:ManyEdges5}
}
\caption{Creating red and green snakelets in the proof of \reflem{Do0-2ManyEdges}.}
\label{Fig:ManyEdges}
\end{figure}

\begin{figure}[htbp] 
\subfloat[]{
\includegraphics[width = 0.35\textwidth]{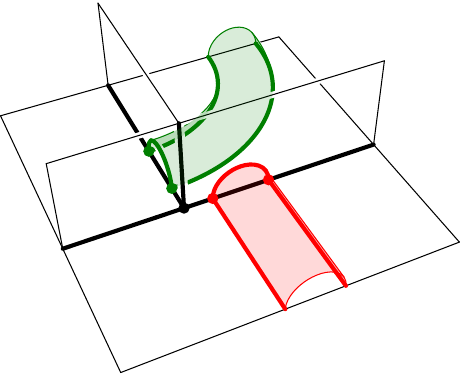}
\label{Fig:DodgeSnakeletHead0}
}
\quad
\subfloat[]{
\includegraphics[width = 0.35\textwidth]{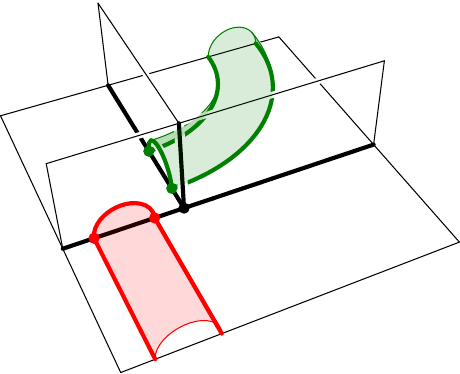}
\label{Fig:DodgeSnakeletHead1}
}

\subfloat[]{
\includegraphics[width = 0.35\textwidth]{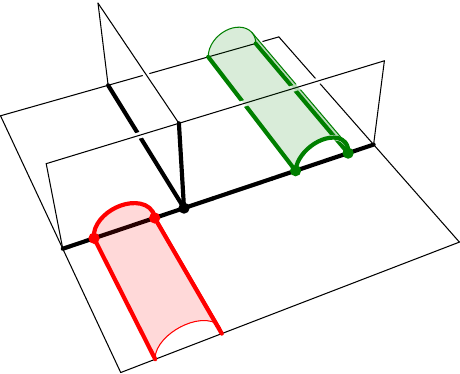}
\label{Fig:DodgeSnakeletHead2}
}
\quad
\subfloat[]{
\includegraphics[width = 0.35\textwidth]{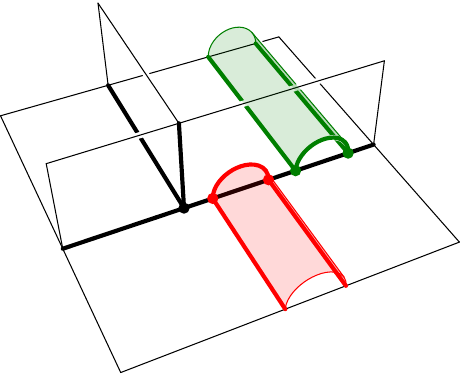}
\label{Fig:DodgeSnakeletHead3}
}
\caption{A regular neighbourhood of the vertex $v_{N-1}$. The red snakelet moves back one step to let the green snakelet past, then moves forward into place again.}
\label{Fig:DodgeSnakeletHead}
\end{figure}

\begin{proof}
Our goal is to provide a set of moves connecting the given foam $\calF$ to a foam with the hypotheses of  \reflem{Do0-2ThreeSides}, apply that lemma, and then provide a set of moves connecting the output of that lemma to the desired foam $\calF[\delta]$.

If $N = 3$ then we apply \reflem{Do0-2ThreeSides} and we are done.
Now suppose that $N > 3$.
Let $A = E_0$ and $B = E_N$.
See \reffig{ManyEdges0}.
Assume for now that $k$ is not equal to either $1$ or $N - 1$.
Using the hypotheses that $e_{k+1}$ is not cyclic, and that $L(E)$ appears once as the label of an outer region for $s$, by \reflem{ImplementVMove} we may apply a V-move, creating a \emph{red snakelet} as shown in \reffig{ManyEdges1}.
Repeatedly applying \reflem{SnakeletVertex} (and again using the hypothesis that $L(E)$ appears once as the label of an outer region for $s$), we move the head of the red snakelet just past $v_{N-1}$ and onto $e_N$.
See \reffig{ManyEdges2}.
Note that since $\phi_M(e_k)$ is distinct from $\phi_M(e_i)$ for $i \neq k$, when moving the head of the red snakelet we never need it to cross its own tail.

Next, using the hypothesis that $e_{k-1}$ is not cyclic, we apply a V-move to create a \emph{green snakelet}, as shown in \reffig{ManyEdges3}. 
We then move the head of the green snakelet just past $v_0$ and onto $e_0$, as shown in \reffig{ManyEdges4}.
These moves on the green snakelet are possible unless the head of the red snakelet blocks us.
If it does, we move the red snakelet back one step onto $e_{N-1}$, apply the desired move to the green snakelet, and then move the red snakelet forward to $e_N$ again.
This is illustrated in \reffig{DodgeSnakeletHead} for the case of moving the head of the green snakelet past a vertex.
We deal with the case of creating the green snakelet with a V-move (when the head of the red snakelet is in the way) in a similar fashion.

With the red and green snakelets in place, the side $s$ of $\delta$ has been reduced to three edges, with the middle edge being made up of a small segment of the original edge $e_k$, together with parts of the one-skeletons of the red and green snakelets.
See \reffig{ManyEdges4}.
This edge is not cyclic because its endpoints are distinct, being vertices on different snakelets.
One can check that the remaining hypotheses of \reflem{Do0-2ThreeSides} also hold, and thus we can do the 0-2 move, building the \emph{blue snakelet}.
The result is illustrated in \reffig{ManyEdges5} in the case that $\phi_M(e_0) \neq \phi_M(e_N)$.

With the blue snakelet in place, we must deconstruct the red and green snakelets. 
This is essentially the reverse of the process we used to create the red and green snakelets and move them into place.
The only difference is that our snakelet heads move around the face $f_s$ instead of $s$.
Note that it is possible for a 0-2 move to introduce a cyclic edge when $\phi_M(e_0) = \phi_M(e_N)$.
However, such a cyclic edge is based at one of the new vertices (on the blue snakelet) created by the 0-2 move, not any of the $v_i$.
In particular, the edges next to $E$ around $f_s$ are not cyclic, so \reflem{ImplementVMove} can be used in reverse to deconstruct the red and green snakelets.

The case that $k=1$ or $k=N-1$ is similar but simpler. 
We need build only one of the ``helper'' snakelets (red or green) in order to reach the hypotheses of \reflem{Do0-2ThreeSides}.
\end{proof}

\section{Distant labels}
\label{Sec:DistantLabels}

We now turn to the problem of implementing a 0-2 move in general. 

\subsection{Locally frozen configurations}
\label{Sec:LocallyFrozen}
As an example, suppose we want to perform a 0-2 move along an arc $\delta$ that connects opposite sides of a hexagonal face $f$. 
Suppose that the outer regions around $f$ alternate between two labels, $a$ and $b$ say. 
Then performing a 0-2 move along $\delta$ creates a new face between regions with labels $a$ and $b$, so does not create an $L$--inessential face.
However, none of the edges of $\bdy f$ are $L$--flippable.
Thus, there is no local 2-3 move that we can use in \reflem{ImplementVMove} to start building snakelets.
Depending on the combinatorics around $f$, there may be no 2-3 or 3-2 move that we can perform anywhere near $f$.
So we may need to start work very far away.
This is why  \refprop{Make0-2} includes the assumption that the image of $L$ is infinite.  

\subsection{Strategy}
\label{Sec:Strategy}

In the remainder of the paper we show how to use a (possibly distant) $L$--flippable edge to perform the 0-2 move by reducing to the case of \reflem{Do0-2ManyEdges}.
In order to do this, we will grow a complementary region $E$, having a ``distant'' label, through the foam to bring it into contact with the target face $f$.
We will use \emph{augmented 2-3 moves} to do this.
See \refsec{Augmented2-3}.
This done, we apply \reflem{Do0-2ManyEdges}.
Then we must undo all of the augmented 2-3 moves.
However, there is a potential obstacle to performing exactly the inverse moves in reverse order;
the 0-2 move along the arc $\phidelta$ alters the foam.
Said another way, our sequence of moves growing $E$ must commute with the 0-2 move along $\phidelta$.

To achieve this goal, we introduce two variants of a new \emph{nature reserve move}, described in \refsec{NatureReserve}. 
These ``protect'' the endpoints of the arc $\phidelta$ from being disturbed by augmented 2-3 moves.
Furthermore, they also protect the snakelet formed by performing a 0-2 move along $\phidelta$ from being disturbed by the same augmented 2-3 moves.

\subsection{Handle structures for foams}

We use the following notion of \emph{handle structures} to organise the argument.

Suppose that $M$ is a compact, connected three-manifold with boundary.
Suppose that $\calF$ is a foam in $M$.
For each vertex $v$ of $\calF$ we choose a regular neighbourhood $\eta(v) \subset M$ of $v$ which we call the \emph{zero-handle} for $v$.
We make the neighbourhoods small enough that the zero-handles are disjoint.
To agree with the notation below we also set $\overline{\eta}(v) = \eta(v)$.

For each edge $e$ of $\calF$ we choose a regular neighbourhood $\eta(e)$ of
\[
e - \bigcup_{v \in \bdy e} \overline{\eta}(v)
\qquad
\mbox{in} 
\qquad 
M - \bigcup_{v \in \bdy e} \overline{\eta}(v) 
\]
which we call the \emph{one-handle} for $e$.
We make the neighbourhoods small enough that the one-handles are disjoint.
We define
\[
\overline{\eta}(e) = \eta(e) \cup \bigcup_{v \in \bdy e} \overline{\eta}(v)
\]

For each face $f$ of $\calF$ we choose a regular neighbourhood $\eta(f)$ of
\[
f - \bigcup_{e \in \bdy{f}}\overline{\eta}(e)
\qquad
\mbox{in} 
\qquad 
M - \bigcup_{e \in \bdy{f}} \overline{\eta}(e)
\] 
which we call the \emph{two-handle} for $f$.
We make the neighbourhoods small enough that the two-handles are disjoint.
We define
\[
\overline{\eta}(f) = \eta(f) \cup \bigcup_{e \in \bdy{f}} \overline{\eta}(e) 
\]

We say that the collection of zero-, one-, and two-handles is a \emph{handle structure} $\eta(\calF)$ for $\calF$.
We say that a foam $\calG$ of $M$ is \emph{carried} by $\calF$ if there is a \emph{carrying function} $C$ from the vertices, edges, and faces of $\calG$ to the vertices, edges, and faces of $\calF$ so that for all cells $c$ of $\calG$ we have:
\begin{enumerate}
\item $c$ lies in $\overline{\eta}(C(c))$, 
\item $c$ does not lie in $\overline{\eta}(d)$ for any model facet $d$ of $C(c)$, and
\item the dimension of $C(c)$ is at most the dimension of $c$. \qedhere
\end{enumerate}

Note that when $\calG$ is carried by $\calF$, the carrying function is unique.

\subsection{Augmented 2-3 moves}
\label{Sec:Augmented2-3}

The main move we use is the following.

\def\mywidth{0.28}

\begin{figure}[htbp]
\subfloat[Before the augmented 2-3 move.]{
\labellist
\hair 2pt \small
\pinlabel $u$ [b] at 125 258
\pinlabel $e$ [r] at 126 184
\pinlabel $v$ [tr] at 126 82
\endlabellist
\includegraphics[width = \mywidth\textwidth]{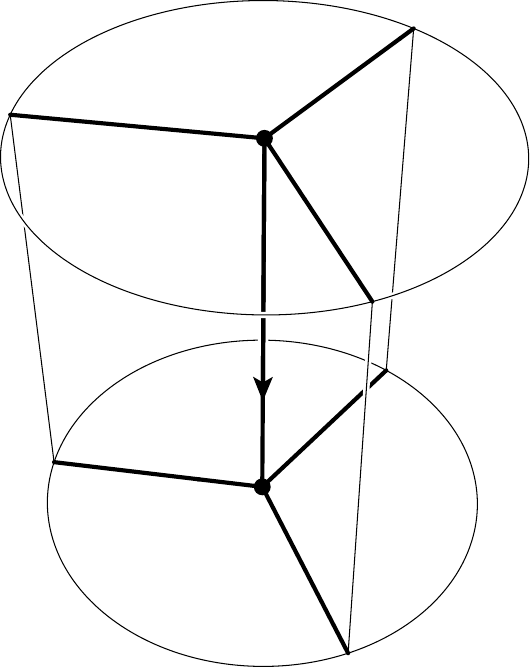}
\label{Fig:Augmented2-3Before}
}
\quad
\subfloat[After doing four V-moves.]{
\labellist
\hair 2pt \small
\pinlabel $f$ at 40 160
\endlabellist
\includegraphics[width = \mywidth\textwidth]{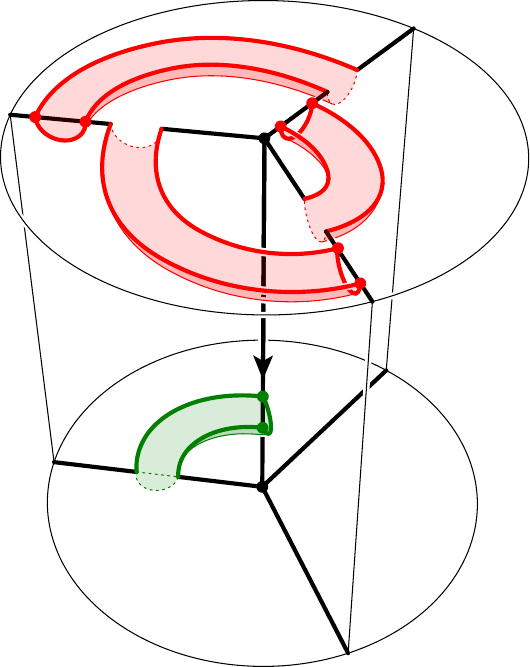}
\label{Fig:Augmented2-3AfterVMoves}
}
\quad
\subfloat[After moving the green snakelet into place.]{
\labellist
\hair 2pt \small
\pinlabel $f'$ at 225 245
\endlabellist
\includegraphics[width = \mywidth\textwidth]{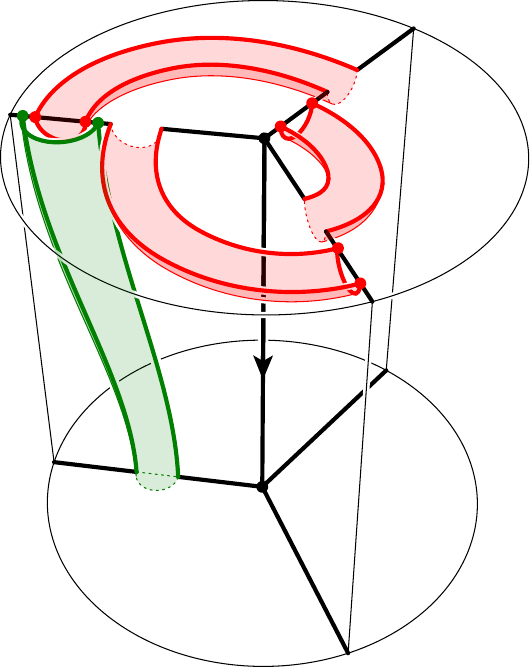}
\label{Fig:Augmented2-3HelperInPosition}
}

\subfloat[After sliding the innermost red snakelet out.]{
\includegraphics[width = \mywidth\textwidth]{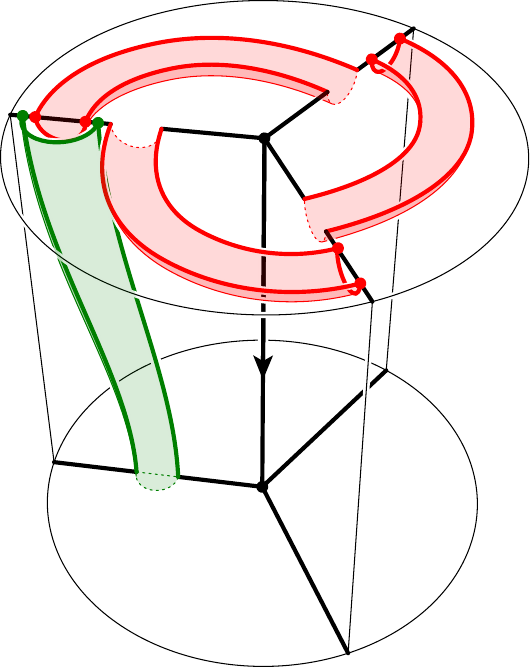}
\label{Fig:Augmented2-3AfterSlide}
}
\quad
\subfloat[After deconstructing the green snakelet.]{
\includegraphics[width = \mywidth\textwidth]{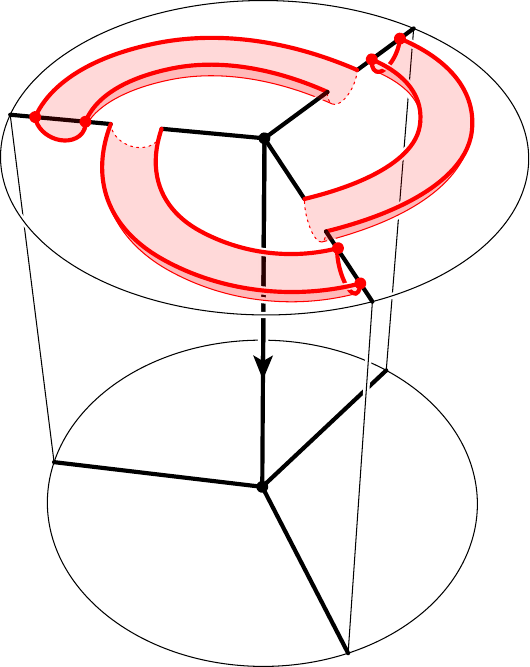}
\label{Fig:ImpossibleStaircase}
}
\quad
\subfloat[After the 2-3 move.]{
\includegraphics[width = \mywidth\textwidth]{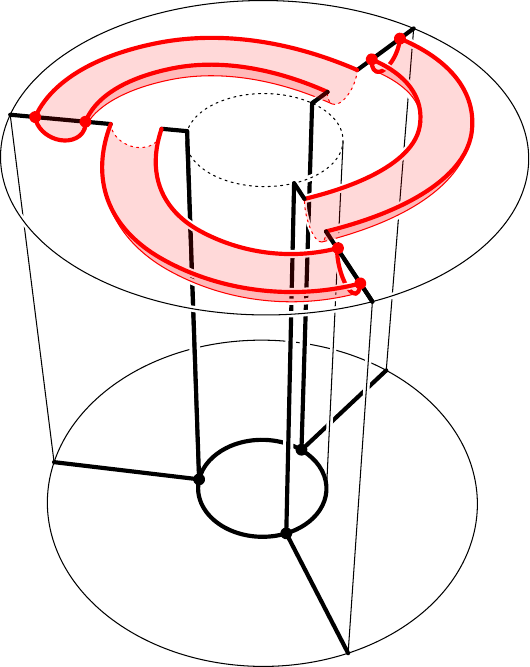}
\label{Fig:Augmented2-3Final}
}
\caption{An augmented 2-3 move.}
\label{Fig:Augmented2-3}
\end{figure}

\begin{definition}
\label{Def:Augmented2-3}
Suppose that $M$ is a compact, connected three-manifold with boundary.
Suppose that $L$ is a labelling of $\Delta_M$.
Suppose that $\calG$ is an $L$--essential foam in $M$.
Suppose that $e$ is an $L$--flippable edge of $\calG$.
Suppose that $e$ is equipped with an orientation, pointing from the vertex $u$ to the vertex $v$.
See \reffig{Augmented2-3Before}.
Note that $u \neq v$ because $e$ is not cyclic.
Fix small regular open neighbourhoods $\calN(e)$, $\calN(u)$, and $\calN(v)$ of $e$, $u$, and $v$ such that $\calN(u)$ and $\calN(v)$ lie within $\calN(e)$.
We perform an \emph{augmented 2-3 move} along $e$ to produce a foam $\calG_e$ as shown in \reffig{Augmented2-3Final}.

We require the following.
\begin{enumerate}
\item The augmented 2-3 move is supported within $\calN(e)$.
\item The three red snakelets shown in \reffig{Augmented2-3Final} are in $\calN(u)$.
\item All cells generated by the 2-3 move that do not have ancestors are in $\calN(v)$.
\item 
\label{Itm:Aug2-3Carried}
Suppose additionally that $\calG$ is carried by some foam $\calF$ with carrying function $C$.
In this case we also require that $\calN(c)$ lies within $\overline{\eta}(C(c))$ as $c$ ranges over $e$, $u$, and $v$. \qedhere
\end{enumerate} 
\end{definition}

\begin{lemma}
\label{Lem:Augmented2-3LEssential}
The augmented 2-3 move along $e$ can be realised by a sequence of 2-3 and 3-2 moves from $\calG$ to $\calG_e$ that pass through $L$--essential foams.
\end{lemma}

\begin{proof}
First, we apply \reflem{ImplementVMove} four times, creating three \emph{red snakelets} and one \emph{green snakelet}.
See \reffig{Augmented2-3AfterVMoves}.
Next we slide the head of the green snakelet around the boundary of the face $f$, using \reflem{SnakeletVertex} three times. 
After a 2-3 and then a 3-2 move (similar to the moves between \reffig{ThreeEdges2b} and \reffig{ThreeEdges3b}), we reach \reffig{Augmented2-3HelperInPosition}, in which the belly of the green snakelet covers the head of one of the red snakelets.
This done, using \reflem{SnakeletVertex} we slide the innermost red snakelet around the face $f'$ to get to \reffig{Augmented2-3AfterSlide}.
The green snakelet ensures that we do not introduce an $L$--inessential face as we do this.
Next, we reverse the movement of the green snakelet by sliding it back around the boundary of $f$.
We undo the V-move, deconstructing the green snakelet.
This takes us to \reffig{ImpossibleStaircase}.
Finally we apply the 2-3 move along $e$, giving the foam $\calG_e$ shown in \reffig{Augmented2-3Final}.
\end{proof}

We collect several useful properties of the augmented 2-3 move.
In particular, augmented 2-3 moves do not create cyclic edges, and so can safely be used to destroy them.
The third property below refers to $\Nature$, a collection of \emph{nature reserve edges}.
See \refsec{ParallelSequences}.

\begin{lemma}
\label{Lem:AugmentedOmnibus}
Suppose that $\calG$ is carried by $\calF$.
Suppose that $\calG$ is $L$--essential.
Suppose that $e$ is an $L$--flippable oriented edge of $\calG$.
Suppose that $\calG_e$ is the result of applying an augmented 2-3 move along $e$.
Then we have the following.

\begin{enumerate}
\item
\label{Itm:AugmentedLEssential}   
The foam $\calG_e$ is $L$--essential.
\item
\label{Itm:AugmentedLFlippable}
The foam $\calG_e$ has an $L$--flippable edge.
\item
\label{Itm:CyclicEdges}
No cyclic edges of $\calG_e$ meet $\calN(e)$. 
\item
\label{Itm:Carried}
The foam $\calG_e$ is carried by $\calF$.
\item 
\label{Itm:MaintainConnectivity}
Suppose that $w$ is a vertex of $\calF$.
Suppose that $\Nature$ is a collection of edges in $\calG^{(1)}$.  
Suppose that $e$ is not in $\Nature$.
Let $\Nature_e$ be the collection of descendants of the edges of $\Nature$ under the augmented 2-3 move.
Suppose that $X = \calG^{(1)} \cap \eta(w) - \Nature$ is connected.
Then $X_e = \calG_e^{(1)} \cap \eta(w) - \Nature_e$ is connected.
\end{enumerate}
\end{lemma}

\begin{proof}
Property \refitm{AugmentedLEssential} follows from \reflem{VMoveL} (or \reflem{Augmented2-3LEssential}) and the fact that $e$ is $L$-flippable.
Each of the three edges created by the final 2-3 move along $e$ is $L$--flippable. 
This gives property \refitm{AugmentedLFlippable}. 

Since the augmented 2-3 move occurs within a ball, all nine vertices it generates are distinct.
By consulting \reffig{Augmented2-3Final} we see that there are no cyclic edges entirely contained within the figure, and no vertex has more than one edge-end leaving the figure.
This gives property \refitm{CyclicEdges}.

Suppose that $C$ is the carrying function for $\calG$ in $\calF$.
Suppose that $e$ was oriented from vertex $u$ to vertex $v$.
Suppose that $d$ is a cell of $\calG_e$.
There are three cases.
First suppose that $d$ is disjoint from $\calN(e)$.
Then $c = d$ is the ancestor of $d$ in $\calG$.
Thus $d$ is carried by the cell $C(c)$ in $\calF$.
Second, suppose that $d$ meets but is not contained in $\calN(e)$.
(There are six edge ends and nine pieces of faces of this type.)
Consulting \reffig{Augmented2-3Final} we see that $d$ is strictly contained in a cell $c$ of $\calG$.
Furthermore, $d$ is carried by $C(c)$.
Finally, suppose that $d$ is contained in $\calN(e)$.
(There are nine vertices, fifteen edges, and seven faces of this type.)
By \refdef{Augmented2-3}\refitm{Aug2-3Carried}, each of these is carried by one of $C(u)$, $C(v)$, or $C(e)$.
Thus we have established \refitm{Carried}. 

Let $c$ be any one of $u$, $v$, or $e$.
Consulting \reffig{Augmented2-3Before} (respectively \reffig{Augmented2-3Final}) we see that $\Nature$ ($\Nature_e$) meets $\calG^{(1)} \cap \calN(c)$ ($\calG_e^{(1)} \cap \calN(c)$) in at most only edge ends.
Again consulting the figures, we see that $\calG^{(1)} \cap \calN(c) - \Nature$ and $\calG_e^{(1)} \cap \calN(c) - \Nature_e$ are each connected.
We now recall the following. 

\begin{fact}
\label{Fac:GraphSwap}
Suppose that $A \cup B$ and $A \cup B'$ are graphs with $A$ being a subgraph of both $A \cup B$ and $A \cup B'$, while $B$, and $B'$ are subgraphs of $A \cup B$ and $A \cup B'$ respectively.
Suppose that $A \cap B = A \cap B'$.
Suppose that $B$ and $B'$ are both connected.
Suppose that $A \cup B$ is connected.
Then $A \cup B'$ is connected.
\end{fact}

The three-ball $\eta(w)$ either: 
\begin{itemize}
\item is disjoint from $\calN(e)$, 
\item contains $\calN(e)$,
\item contains $\calN(u)$ but not $\calN(v)$,
\item contains $\calN(v)$ but not $\calN(u)$, or
\item contains $\calN(u) \sqcup \calN(v)$ but not $\calN(e)$.
\end{itemize}

Thus the number of components of $\eta(w) \cap \calN(e)$ is either zero, one, or two.
If $\eta(w) \cap \calN(e)$ is empty then $X = X_e$ and we are done.
Suppose instead that $\eta(w)$ contains $\calN(e)$.
We are given that $X$ is connected.
By taking closures, $X$ and $X_e$ become graphs.
Take $A = X - \calN(e) = X_e - \calN(e)$, take $B = X \cap \calN(e) = \calG^{(1)} \cap \calN(e) - \Nature$, and take $B' = X_e \cap \calN(e) =  \calG_e^{(1)} \cap \calN(e) - \Nature_e$.
By the above argument, $B$ and $B'$ are both connected.
Then $A \cup B =X$, $A \cup B' = X_e$, and $A \cap B = A \cap B'$.
Applying \reffac{GraphSwap}, we get that $X_e$ is connected.

The argument is similar in the other cases, replacing $\calN(e)$ with $\calN(u)$, $\calN(v)$, or their disjoint union.
(In the latter case, we apply \reffac{GraphSwap} twice.)
Thus we obtain \refitm{MaintainConnectivity}.
\end{proof}

\begin{remark}
Carrying only places restrictions on the zero-, one-, and two-dimensional cells.
The three-dimensional complementary regions (components of $M - \calG_e$) are not similarly restricted. 
Indeed the main purpose of the augmented 2-3 move is to grow complementary regions throughout the universal cover.
\end{remark}

\subsection{Nature reserve moves}
\label{Sec:NatureReserve}

In order to define a nature reserve move, we make the following assumptions:
Suppose that $M$ is a compact, connected three-manifold with boundary.
Suppose that $L$ is a labelling of $\Delta_M$.
Suppose that $\calG$ is an $L$--essential foam in $M$.
Suppose that $e$ is an $L$--flippable edge of $\calG$.
Suppose that $p$ is one of the endpoints of $e$ (since $e$ is not cyclic it has distinct endpoints).
Let $\calN(p)$ be a small regular neighbourhood of $p$.
If $\calG$ is carried by some foam $\calF$ with carrying function $C$ then we further require that $\calN(p)$ lies within $\eta(C(p))$.

Suppose that $\phidelta$ is an arc in a face of $\calG$ with at least one endpoint on $e$.
Suppose that $\calG[\delta]$ is $L$--essential and has an $L$--flippable edge. 
There are two possibilities. 
Either the arc $\phidelta$ has \emph{exactly one} endpoint $d$ on $e$, or it has \emph{both} endpoints, $d$ and $d'$, on $e$.

\subsubsection{Singleton nature reserve move}
\label{Sec:SingletonNatureReserve}

\begin{figure}[htbp]
\subfloat[Part of the foam $\calG'$. The arc $\phidelta$ has an endpoint $d$ on $e$.]{
\labellist
\hair 2pt \small
\pinlabel $p$ [b] at 128 257
\pinlabel $d$ [r] at 126 234
\pinlabel $e$ [r] at 126 124
\pinlabel $\phidelta$ [t] at 165 256
\endlabellist
\includegraphics[width = 0.41\textwidth]{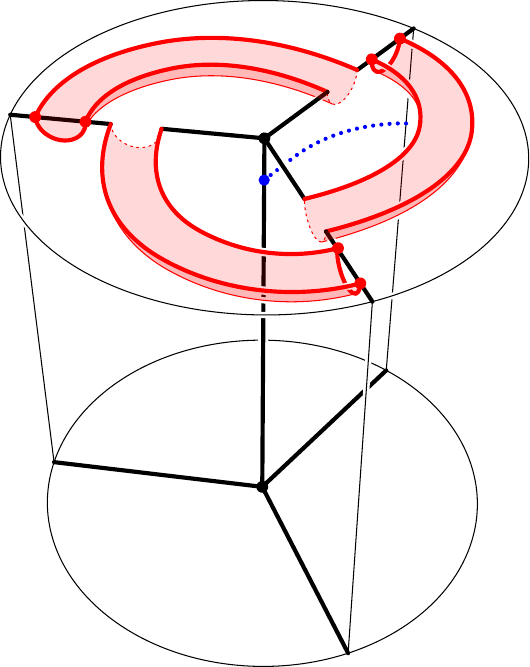}
\label{Fig:ImpossibleStaircaseDelta}
}
\quad
\subfloat[Part of the foam $\calG_p$. A nature reserve edge protects $d$.]{
\includegraphics[width = 0.41\textwidth]{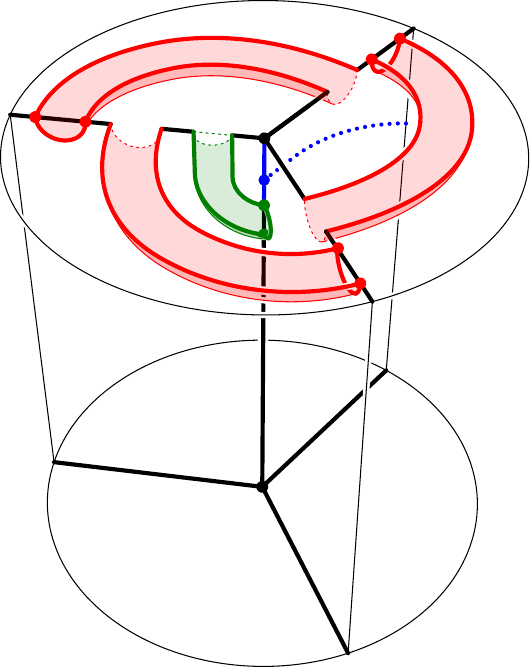}
\label{Fig:ImpossibleStaircaseNatureReserve}
}

\subfloat[Close up of \reffig{ImpossibleStaircaseDelta}.]{
\labellist
\hair 2pt \small
\pinlabel $p$ [tl] at 45 62
\pinlabel $\phidelta$ [r] at 124 80
\pinlabel $d$ [t] at 120 57
\pinlabel $e$ [b] at 185 57
\endlabellist
\includegraphics[width = 0.41\textwidth]{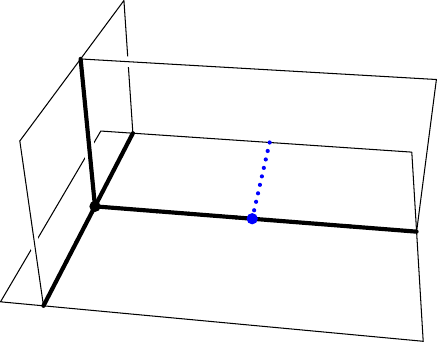}
\label{Fig:NatureReserveNoRed}
}
\quad
\subfloat[Close up of \reffig{ImpossibleStaircaseNatureReserve}.]{
\labellist
\hair 2pt \small
\pinlabel $e'$ [b] at 140 60
\pinlabel $e''$ [t] at 156 56
\pinlabel $e'''$ [b] at 186 57
\endlabellist
\includegraphics[width = 0.41\textwidth]{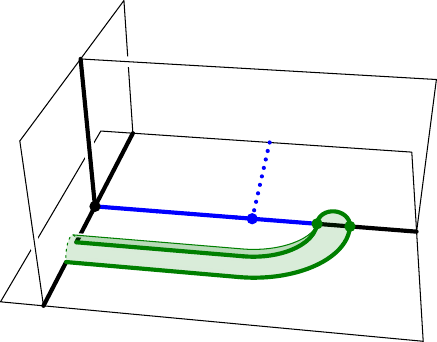}
\label{Fig:NatureReserveRed}
}

\subfloat[Part of the foam \mbox{$\calG[\delta]$}. ]{  
\labellist
\hair 2pt \small
\pinlabel $p$ [tl] at 45 62
\endlabellist
\includegraphics[width = 0.41\textwidth]{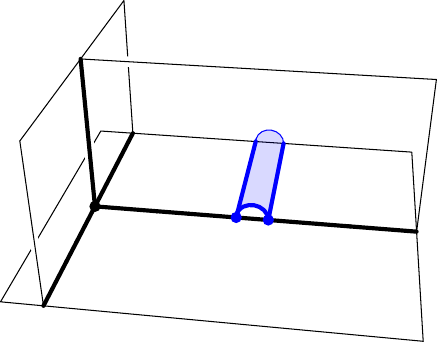}
\label{Fig:NatureReserveGreenNoRed}
}
\quad
\subfloat[Part of the foam \mbox{$\calG_p[\delta]$}.]{
\includegraphics[width = 0.41\textwidth]{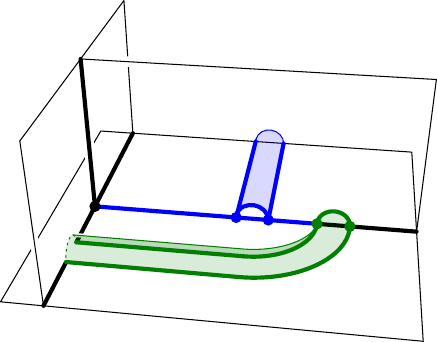}
\label{Fig:NatureReserveGreenRed}
}
\caption{Introducing a nature reserve edge $e'$. In \reffig{NatureReserveGreenNoRed}, we are so close to $p$ that this is also a correct picture of \mbox{$\calG'[\delta]$}.}
\label{Fig:NatureReserveEasy}
\end{figure}

Suppose that the arc $\phidelta$ has precisely one endpoint $d$ on $e$.

\begin{definition}
\label{Def:SingletonNatureReserveMove}
We perform \emph{singleton nature reserve move} at $p$ along $e$ to produce a foam $\calG_p$.
This foam is shown in \reffig{ImpossibleStaircaseNatureReserve}.
We properly isotope $\phidelta$ within $\phif$ to move $d$ along $e$ so that $d$ lies within $\calN(p)$.
We introduce three \emph{red snakelets} and one \emph{green snakelet}, all contained within $\calN(p)$.
The green snakelet lies on a (descendant of a) face $g$ of $\calG \cap \calN(p)$ so that 
\begin{itemize}
\item 
$g$ is incident to the interior of $e$,
\item 
$g$ does not intersect the interior of $\phidelta$, and
\item
along $e$, the endpoint $d$ lies between $p$ and the vertices on the green snakelet.
\end{itemize}
(There are two possibilities for $g$.)
See \reffig{NatureReserveRed}.
In $\calG_p$, the edge $e$ has been split into three segments.
The first of these is the \emph{nature reserve edge}, $e'$, which ``protects'' the endpoint $d$ of $\phidelta$.
The second, with interior incident to the green snakelet, is $e''$ say.
The third and last is the \emph{remainder}, $e'''$ say.
See \reffig{NatureReserveRed}.
\end{definition}

\begin{lemma}
\label{Lem:SingletonNatureReserveMove}
For either choice of vertex $p$, the singleton nature reserve move at $p$ along $e$ can be realised by a sequence of 2-3 and 3-2 moves from $\calG$ to $\calG_p$ with all foams being $L$--essential.
\end{lemma}

\begin{proof}
Since $e$ is $L$--flippable, we follow the steps of \refsec{Augmented2-3}, within $\calN(p)$, to reproduce the construction shown in Figures~\ref{Fig:Augmented2-3Before} to~\ref{Fig:ImpossibleStaircase}, (with $e$ oriented away from $p$).
We call the resulting foam $\calG'$.
See \reffig{ImpossibleStaircaseDelta}.

We perform a V-move at the vertex $p$ (shown in closeup in \reffig{NatureReserveNoRed}) to produce the foam $\calG_p$; 
the hypotheses of \reflem{ImplementVMove} are satisfied because $e$ is $L$--flippable.
\end{proof}

\begin{lemma}
\label{Lem:SingletonNatureReserveMoveDelta}
There is a choice of vertex $p$ so that there is a sequence of 2-3 and 3-2 moves from $\calG[\delta]$ to $\calG_p[\delta]$ with all foams being $L$--essential.
\end{lemma}

\begin{proof}
Suppose that $\cover{e}$ is a lift of $e$.
Let $P$ and $Q$ be the regions of $\cover{M} - \cover{\calG}$ that meet the endpoints of $\cover{e}$ but not the interior.
Since $e$ is $L$--flippable, $L(P) \neq L(Q)$. 
After performing the 0-2 move along $\phidelta$, \reffig{NatureReserveNoRed} becomes \reffig{NatureReserveGreenNoRed}.
Here we see an additional region $R$ of $\cover{M} - \cover{\calG}[\delta]$ that meets the interior of $\cover{e}$.
Within the figure, $R$ is inside the snakelet.
Since $L(P) \neq L(Q)$, the label $L(R)$ is different from at least one of $L(P)$ and $L(Q)$.
Let $\cover{p}$ be an endpoint of $\cover{e}$ incident to a region (either $P$ or $Q$) which does not have label $L(R)$.
Let $p = \phi_M(\cover{p})$ be the image under the covering map.

We deduce that the edge between $p$ and the head of the blue snakelet is $L$--flippable.
Thus following the steps of \refsec{Augmented2-3} as before, we can get from $\calG[\delta]$ to $\calG'[\delta]$.
Again by our choice of $p$ we can apply \reflem{ImplementVMove} and \reflem{SnakeletVertex} to build and then slide the head of the green snakelet past the head of the blue snakelet built along $\phidelta$.
This done, we have reached $\calG_p[\delta]$. 
Again, see \reffig{NatureReserveGreenRed}.
\end{proof}

\subsubsection{Pair nature reserve move}
\label{Sec:PairNatureReserve}

Suppose that the arc $\phidelta$ has both endpoints $d$ and $d'$ on $e$.

\begin{definition}
We perform a \emph{pair nature reserve move} at $p$ along $e$ to produce a foam $\calG_p$.
This foam is shown in \reffig{ImpossibleStaircasePairNatureReserve}.
We properly isotope $\phidelta$ within $\phif$ to move $d$ and $d'$ along $e$ so that both lie within $\calN(p)$.
Exactly as in \refdef{SingletonNatureReserveMove}, we introduce three \emph{red snakelets} and one \emph{green snakelet}, all contained within $\calN(p)$.
We arrange matters as in \refdef{SingletonNatureReserveMove}, except that now 
\begin{itemize}
\item
along $e$, the endpoints $d$ and $d'$ lie between $p$ and the vertices on the green snakelet.
\end{itemize}
(In \refdef{SingletonNatureReserveMove} there were two possibilities for $g$; now there is only one.)
See \reffig{NatureReserve2PathRed}.
As in \refdef{SingletonNatureReserveMove}, in $\calG_p$ the edge $e$ has been split into three segments. 
These are the \emph{nature reserve edge} $e'$ (containing $d$ and $d'$), the edge $e''$, whose interior is incident to the green snakelet, and the \emph{remainder} $e'''$. 
See \reffig{NatureReserve2PathRed}.
\end{definition}

The proof of the following lemma is identical to that of \reflem{SingletonNatureReserveMove} and we omit it.

\begin{lemma}
\label{Lem:PairNatureReserveMove}
For either choice of vertex $p$, the pair nature reserve move at $p$ along $e$ can be realised by a sequence of 2-3 and 3-2 moves from $\calG$ to $\calG_p$ with all foams being $L$--essential. \qed
\end{lemma}


\begin{figure}[htbp]
\subfloat[Part of the foam $\calG'$. The arc $\phidelta$ has both endpoints, $d$ and $d'$ on $e$.]{
\labellist
\hair 2pt \small
\pinlabel $p$ [b] at 128 257
\pinlabel $d$ [r] at 123 244
\pinlabel $d'$ [r] at 126 224
\pinlabel $e$ [r] at 126 124
\pinlabel $\phidelta$ [t] at 165 256
\pinlabel $\phidelta$ [tr] at 143 218
\endlabellist
\includegraphics[width = 0.41\textwidth]{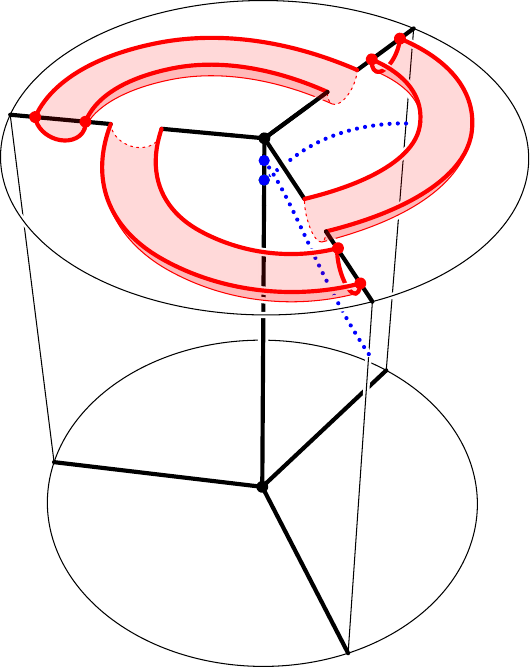}
\label{Fig:ImpossibleStaircasePairDelta}
}
\quad
\subfloat[Part of the foam $\calG_p$. A nature reserve edge protects $d$ and $d'$.]{
\includegraphics[width = 0.41\textwidth]{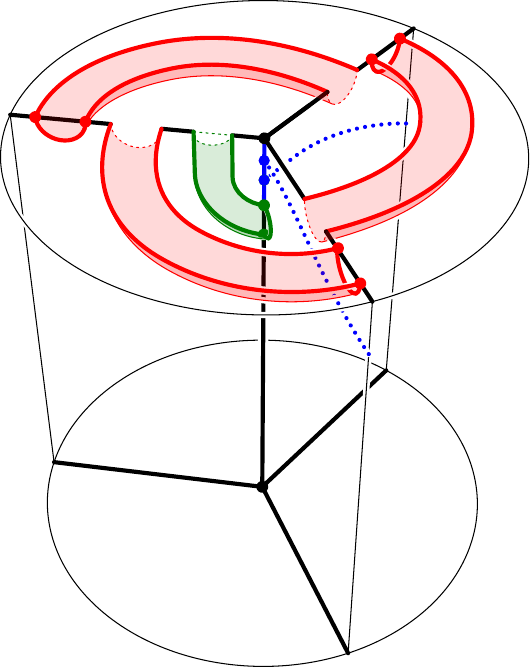}
\label{Fig:ImpossibleStaircasePairNatureReserve}
}

\subfloat[Close up of \reffig{ImpossibleStaircasePairDelta}.]{
\labellist
\hair 2pt \small
\pinlabel $p$ [tl] at 45 62
\pinlabel $\phidelta$ [r] at 124 80
\pinlabel $e$ [b] at 185 57
\pinlabel $\phidelta$ [l] at 84 117
\pinlabel $d$ [t] at 83 60
\pinlabel $d'$ [t] at 121 57
\endlabellist
\includegraphics[width = 0.41\textwidth]{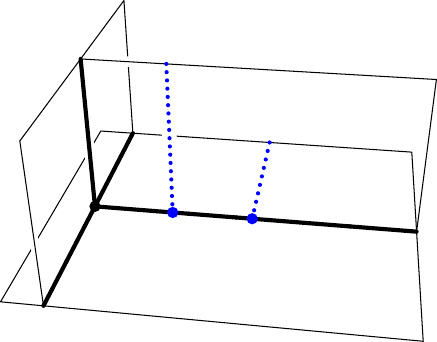}
\label{Fig:NatureReserve2PathNoRed} 
} 
\quad
\subfloat[Close up of \reffig{ImpossibleStaircasePairNatureReserve}.]{
\labellist
\hair 2pt \small
\pinlabel $e'$ [b] at 140 60
\pinlabel $e''$ [t] at 156 56
\pinlabel $e'''$ [b] at 186 57
\endlabellist
\includegraphics[width = 0.41\textwidth]{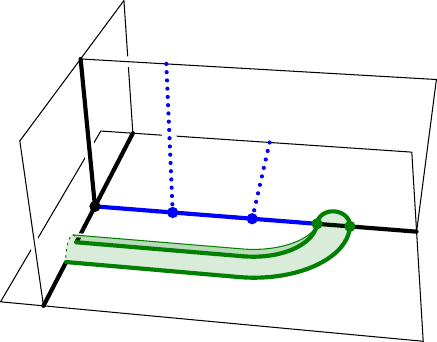}
\label{Fig:NatureReserve2PathRed}
}

\subfloat[\reffig{NatureReserve2PathNoRed} with the snakelet generated by the 0-2 move.
This is part of the foam {$\calG'[\delta]$}.]{
\labellist
\hair 2pt \small
\pinlabel $\epsilon_3$ [t] at 100 37
\endlabellist
\includegraphics[width = 0.41\textwidth]{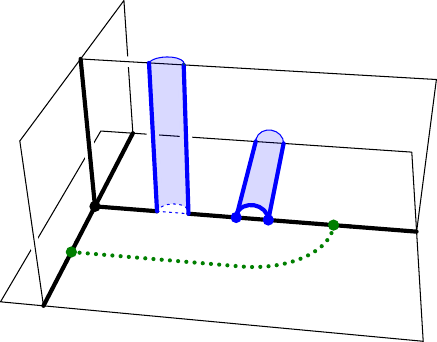}
\label{Fig:NatureReserve2GreenNoRed}
}
\quad
\subfloat[\reffig{NatureReserve2PathRed} with the snakelet generated by the 0-2 move. 
This is part of the foam {$\calG_p[\delta]$}.]{
\includegraphics[width = 0.41\textwidth]{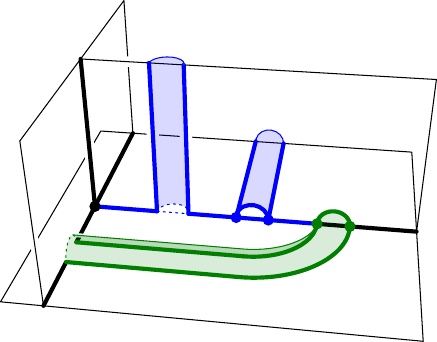}
\label{Fig:NatureReserve2GreenRed}
}
\caption{Introducing a nature reserve (the segment of $e$ drawn in blue) in the case that both ends of the arc $\phidelta$ lie on the same edge $e$.}
\label{Fig:NatureReserveHard}
\end{figure}

We do not yet have the tools to prove the following lemma.
We defer its proof to \refsec{BothEndpointsOneEdge}.

\begin{lemma}
\label{Lem:PairNatureReserveMoveDelta}
For either choice of $p$ the following holds.
Suppose that $\calG_p$ is the result of applying the pair nature reserve move at $p$ along $e$.
Then there is a sequence of 2-3 and 3-2 moves from $\calG[\delta]$ to $\calG_p[\delta]$ with all foams being $L$--essential.
\end{lemma}

\subsubsection{Properties of nature reserve moves}

We collect several useful properties of the singleton and pair nature reserve moves.

\begin{remark}
\label{Rem:Remainder}
The five regions meeting the closure of $e'''$ in either nature reserve move are the same as the five regions meeting the closure of $e$.
Also, one endpoint of $e'''$ lies on the head of the green snakelet while the other lies outside of $\calN(p)$ (because $e$ was not cyclic).
Thus $e'''$ is not cyclic.
Thus in both cases the remainder edge $e'''$ is $L$--flippable.
\end{remark}

The proof of the following lemma is very similar to the proof of \reflem{AugmentedOmnibus} and we omit it.

\begin{lemma}   
\label{Lem:NatureReserveOmnibus}
Suppose that $\calG$ is carried by $\calF$.
Suppose that $\calG$ is $L$--essential.
Suppose that $e$ is an $L$--flippable edge of $\calG$.
Fix an end $p$ of $e$.
Suppose that $\calG_p$ is the result of applying a (singleton or pair) nature reserve move at $p$ along $e$, producing a nature reserve edge $e'$.
Then we have the following.
\begin{enumerate}
\item
\label{Itm:NatureLEssential}   
The foam $\calG_p$ is $L$--essential.
\item
\label{Itm:NatureLFlippable}
The foam $\calG_p$ has an $L$--flippable edge.
\item
\label{Itm:NatureCyclicEdges}
No cyclic edges of $\calG_p$ meet $\calN(p)$. 
\item
\label{Itm:NatureCarried}
The foam $\calG_p$ is carried by $\calF$.
\item 
\label{Itm:NatureMaintainConnectivity}
Suppose that $w$ is a vertex of $\calF$.
Suppose that $\Nature$ is a collection of edges in $\calG^{(1)}$.
Suppose that $e$ is not in $\Nature$.
Let $\Nature_p$ be the collection of descendants of the edges of $\Nature$ under the singleton nature reserve move, union with $e'$.
Suppose that $\calG^{(1)} \cap \eta(w) - \Nature$  is connected.
Then $\calG_p^{(1)} \cap \eta(w) - \Nature_p$ is connected. \qed
\end{enumerate}
\end{lemma}

\section{Parallel sequences}
\label{Sec:ParallelSequences}

Suppose that $\calF$ is the given $L$--essential foam.
Suppose that $\phidelta$ is the given arc.
Performing a 0-2 move along $\phidelta$ produces the foam $\calF[\delta]$.
In this section, we give the proof of \refprop{Make0-2}.
That is, we give a sequence of 2-3 and 3-2 moves from $\calF$ to $\calF[\delta]$ with all foams being $L$--essential.

We deal with $\calF$ and $\calF[\delta]$ in parallel, working in from both ends of our eventual sequence of 2-3 and 3-2 moves.
That is, we produce a sequence $\calF = \calF_0, \calF_1, \ldots, \calF_K$, where $\calF_i$ and $\calF_{i+1}$ (respectively $\calF_i[\delta]$ and $\calF_{i+1}[\delta]$) are related by a sequence of 2-3 and 3-2 moves through $L$--essential foams.
Furthermore, $F_K$ and $F_K[\delta]$ are the input and output foams of \reflem{Do0-2ManyEdges}.

All foams $\calF_i$ are carried by $\calF$ with carrying function $C_i$, say.
To each foam $\calF_i$ we associate a list $\Visited_i$ of \emph{visited} vertices of $\calF$.
We also maintain a list $\Nature_i$ of zero, one, or two edges of $\calF_i$; these are the \emph{nature reserve edges}.
We abuse notation and refer to the descendant of $f$ in $\calF_i$ as $f$. 

We break the sequence of moves into two stages.
\begin{itemize}
\item
In the \emph{loosening stage}, taking us from $\calF_0$ to $\calF_J$, we
apply loosening moves along all edges of $\calF_0$, destroying any and all cyclic edges.
Furthermore, after the loosening stage the set $\Nature_J$ contains either one or two nature reserve edges, one for each edge containing an endpoint of $\phidelta$.
\item
In the \emph{contacting stage}, taking us from $\calF_J$ to $\calF_K$, we bring a ``distant'' region into contact with the face~$f$.
\end{itemize}

\subsection{Loosening move}
\label{Sec:LooseningMove}

\begin{definition}
Suppose that $\calF_i$, $C_i$, $\Visited_i$, and $\Nature_i$ are the given foam, with its carrying function, list of visited vertices, and nature reserve edges.
Suppose that $e_i$ is an oriented and $L$--flippable edge of $\calF_i$ which is not in $\Nature_i$.
Then the \emph{loosening move} along $e_i$ produces $\calF_{i+1}$,  $C_{i+1}$, $\Visited_{i+1}$, and $\Nature_{i+1}$, as follows.
\begin{itemize}
\item
Suppose that $e_i$ does not meet an endpoint of $\phidelta$.
Then we perform an augmented 2-3 move along $e_i$.
\item
Suppose instead that $e_i$ meets one or both endpoints of $\phidelta$.
First, we apply a nature reserve move (singleton or pair, as appropriate) to $\calF_i$ along $e_i$.
Let $e'''$ be the remainder edge produced by the nature reserve move. 
We orient $e'''$ in the same direction as $e_i$.
Second, we apply an augmented 2-3 move along the remainder edge $e'''$. 
(Note that the augmented 2-3 move along $e'''$ is possible because $e'''$ is $L$--flippable by \refrem{Remainder}.)
\end{itemize}

\noindent
The carrying function $C_{i+1}$ is given by \reflem{AugmentedOmnibus}\refitm{Carried} or \reflem{NatureReserveOmnibus}\refitm{NatureCarried}.

Let $v_i$ be the vertex of $\calF_i$ at the terminal point of $e_i$.
We set $\Visited_{i+1} = \Visited_i \cup \{C_i(v_i)\}$.

\begin{itemize}
\item Suppose that $e_i$ does not meet an endpoint of $\phidelta$.
Since $e_i$ is not in $\Nature_i$, every edge of $\Nature_i$ has a unique descendant under the augmented 2-3 move.
We place this unique descendant in $\Nature_{i+1}$.

\item Suppose that $e_i$ meets one or both endpoints of $\phidelta$.
Since $e_i$ is not in $\Nature_i$, every edge of $\Nature_i$ has a unique descendant under both the nature reserve move and the augmented 2-3 move.
We place this unique descendant in $\Nature_{i+1}$, together with the new nature reserve edge $e'$. \qedhere
\end{itemize}
\end{definition}

To either $\calF_i$ or $\calF_{i+1}$ we may apply the 0-2 move along $\phidelta$ to obtain $\calF_i[\delta]$ or $\calF_{i+1}[\delta]$ respectively.
We record this information in \refdia{CommutingDiagram}, together with a dashed arrow connecting $\calF_i[\delta]$ to $\calF_{i+1}[\delta]$.
Depending on the number of endpoints of $\phidelta$ on $e_i$, a subcollection of Lemmas \ref{Lem:Augmented2-3LEssential}, \ref{Lem:SingletonNatureReserveMove}, \ref{Lem:SingletonNatureReserveMoveDelta},
\ref{Lem:PairNatureReserveMove}, and
\ref{Lem:PairNatureReserveMoveDelta} proves the following.

\begin{corollary}
\label{Cor:CommutingDiagram}
For each horizontal arrow in \refdia{CommutingDiagram} there is a sequence of 2-3 and 3-2 moves with all foams being $L$--essential. \qed
\end{corollary}

\begin{equation}
\label{Dia:CommutingDiagram}
    \begin{tikzcd}
        \calF_i \arrow[maps to]{r} \arrow[maps to, "\mbox{\scriptsize0-2}"]{d} & \calF_{i+1} \arrow[maps to, "\mbox{\scriptsize0-2}"]{d}  \\
        \calF_i[\delta] \arrow[dashed, maps to]{r}  & \calF_{i+1}[\delta]
    \end{tikzcd}
\end{equation}

Thus the loosening move commutes with the 0-2 move along $\phidelta$.
Note that we have not yet established \reflem{PairNatureReserveMoveDelta} (the case in which $\phidelta$ has both endpoints on $e_i$).
The proof of \reflem{PairNatureReserveMoveDelta} in fact requires \refcor{CommutingDiagram} in the case in which $\phidelta$ has zero or one endpoint on $e_i$.

\subsection{Loosening stage}
\label{Sec:LooseningStage}

From here and until \refsec{BothEndpointsOneEdge}, we will assume the following. 

\begin{hypothesis}
\label{Hyp:OneEndpoint}
The endpoints of $\phidelta$ lie on distinct edges of $\calF$.
\end{hypothesis}

\noindent
(Under this hypothesis, as discussed in \refsec{NatureReserve}, we do not use the pair nature reserve move.
\refhyp{OneEndpoint} is used only when a loosening move is performed along an edge containing an endpoint of $\phidelta$.)
We now recursively choose the sequence of edges along which we apply loosening moves.

We first deal with the base case.
We set $\calF_0 = \calF$, and we take both $\Visited_0$ and $\Nature_0$ to be empty.
We now choose $e_0$ to be any $L$--flippable edge of $\calF_0$.
Such an edge exists by the hypotheses of \refprop{Make0-2}.
We arbitrarily orient $e_0$.
Let $u_0$ and $v_0$ be the initial and terminal vertices of $e_0$.
We apply a loosening move along $e_0$.
This takes us from $\calF_0$ to $\calF_1$.
The set of visited vertices becomes $\Visited_1 = \{v_0\}$.
If $e_0$ does not meet an endpoint of $\phidelta$ then $\Nature_1 = \Nature_0$ is empty.
If $e_0$ contains an endpoint of $\phidelta$ then $\Nature_1 = \{ e'_0 \}$, where $e'_0$ is the nature reserve edge generated by the nature reserve move.

Our next step is to choose $e_1$ to be one of the three edges of $\calF_1$ that intersect $\eta(e_0)$.
One endpoint of $e_1$ lies in $\eta(u_0)$ while the other lies in $\eta(v_0)$.
We orient $e_1$ from the latter to the former; we now apply a loosening move (necessarily an augmented 2-3 move) along $e_1$.
This takes us from $\calF_1$ to $\calF_2$, and $\Visited_2 = \{u_0, v_0\}$.
Here it is not possible that $e_1$ contains an endpoint of $\phidelta$ so we set $\Nature_2 = \Nature_1$.

We now deal with the inductive step.
We are given a foam $\calF_i$ with its list of visited vertices $\Visited_i$ and nature reserve edges $\Nature_i$.
The induction hypothesis tells us that $\calF_i$ is $L$--essential.
There are now two possibilities.
Either there is an edge of $\calF_i$ which is a descendant of an edge of $\calF_0$, or not.
In the latter case we set $J=i$ and we are done with the loosening stage.

Suppose instead that we have such an edge.
We choose this edge to be $e_i$.
(Thus $C_i(e_i)$ is an edge rather than a vertex.) 
Since the one-skeleton of $\calF$ is connected, we may assume (possibly by choosing a different such edge $e_i$) that $e_i$ has at least one endpoint, $u_i$ say, in a zero-handle $\eta(C_i(u_i))$, where $C_i(u_i) \in \Visited_i$.
We will now grow a nearby region so that $e_i$ becomes $L$--flippable.
We will then perform a loosening move along $e_i$.

Set $w = C_i(u_i)$.
Pick a lift $\cover{w}$ of $w$; let $\eta(\cover{w})$ be the corresponding lift of $\eta(w)$.

\begin{claim}
\label{Clm:FiveRegions}
The zero-handle $\eta(\cover{w})$ intersects regions of $\cover{M} - \cover{\calF}_i$ having at least five labels. 
Furthermore, for all vertices $u$ in $\calF_i \cap \eta(w)$, we have that $u$ is not incident to a cyclic edge in $\calF_i$.
\end{claim}

\begin{proof}
The loosening move that initially added $w$ to the visited list $\Visited_j$ ($j \leq i$) added a fifth region to $\eta(\cover{w}) - \cover{\calF}_j$.
See \reffig{Augmented2-3Final}.
These regions have distinct labels because $\calF_j$ is $L$--essential and the regions are pairwise adjacent.
The second statement follows by \reflem{AugmentedOmnibus}\refitm{CyclicEdges} and \reflem{NatureReserveOmnibus}\refitm{NatureCyclicEdges}.
\end{proof}

Thus $e_i$ is not cyclic.
Let $v_i$ be the other endpoint of $e_i$.
(It is possible that $C_i(u_i) = C_i(v_i)$.)
Let $\cover{e_i}$ be a lift of $e_i$ meeting $\eta(\cover{w})$.
Let $\cover{u_i}$ and $\cover{v_i}$ be the corresponding lifts of $u_i$ and $v_i$.

\subsection{Tunnelling through a zero-handle}
\label{Sec:Tunnelling}

Since $w = C_i(u_i)$ lies in $\Visited_i$, the above \refclm{FiveRegions} tells us that the zero-handle $\eta(C_i(\cover{u_i}))$ of $\cover{\calF}$ intersects regions having at least five labels. 
In particular, it contains a region $E$ with label different from the labels of the four regions incident to $\cover{v_i}$.
By \reflem{AugmentedOmnibus}\refitm{MaintainConnectivity}, \reflem{NatureReserveOmnibus}\refitm{NatureMaintainConnectivity}, and induction, there is a path in the one-skeleton of $\cover{\calF}_i \cap \eta(C_i(\cover{u_i}))$ from $E$ to $\cover{u_i}$ that does not pass through any edge of $\Nature_i$.
Take a minimal such path $\gamma$ and orient it towards $\cover{u_i}$.
Orient $e_i$ from $u_i$ to $v_i$ and form $\gamma'$ by adding $\cover{e}_i$ to the end of $\gamma$.
Note that the edges containing endpoints of $\phidelta$ are either in $\Nature_i$ or they are descendants of edges of $\calF$, in which case they are carried by edges of $\calF$.
In either case, no endpoint of $\phidelta$ lies on $\gamma$.

We repeatedly apply loosening moves along $\gamma'$, creating foams $\calF_{i+1}$, $\calF_{i+2}$, and so on, until we perform a loosening move along (a descendant of) $e_i$.
Strictly speaking, this requires another recursive construction with the following hypotheses.
\begin{itemize}
\item Each remaining edge of $\gamma'$ has a descendant in each foam ($\gamma$ lies in the ball $\eta(C_i(\cover{u}))$ and by the choice of $e_i$).
\item The edges of $\gamma'$ are not cyclic (Lemmas~\ref{Lem:AugmentedOmnibus}\refitm{CyclicEdges} and~\ref{Lem:NatureReserveOmnibus}\refitm{NatureCyclicEdges}).
\item The label $L(E)$ appears as exactly one of the five labels incident to each edge (by the minimality of $\gamma$ and the definition of the loosening move).
\end{itemize}

This completes the recursive step and so reduces the number of edges which are descendants of edges in $\calF$.

When the loosening stage completes, the resulting foam $\calF_J$ contains (by \refhyp{OneEndpoint}) two nature reserve edges.
Furthermore, no edge of $\calF_J$ is cyclic.
This follows since the induction halts when no edge of $\calF_J$ is a descendant of an edge of $\calF_0$.
Thus every such edge at some point is destroyed by either an augmented 2-3 move or a nature reserve move. 
By \reflem{AugmentedOmnibus}\refitm{CyclicEdges} and \reflem{NatureReserveOmnibus}\refitm{NatureCyclicEdges}, these moves destroy all cyclic edges.

\subsection{Contacting stage}
\label{Sec:BringDistant}
The face $f$ (that is, its descendants) survives all loosening moves.
The same holds for the arc $\delta$.
Let $s$ and $s'$ be the two sides of $\delta$ (see \refdef{Sides}).
Consider the set of complementary regions in $\cover{\calF}_J$ that are incident to $f$, unioned with the sets of complementary regions in $\cover{\calF}_J[\delta]$ that are incident to $f_s$ and $f_{s'}$ (see \refdef{FaceFromSide}).
The resulting union is finite.
Since $L$ has infinite image, there is a label $\ell$ different from the labels of all regions in that union.
We choose, in the one-skeleton of $\cover{\calF}_J$ minus $\cover\Nature_J$, a minimal path $\gamma$ connecting a vertex of a region $E$, with $L(E) = \ell$, to a vertex of $f$. 
Our goal is to bring $E$ into contact with $f$.
To do this we extend our sequence of foams by performing further augmented 2-3 moves.

Fix a handle structure $\eta'$ for $\calF_J$.
The foams $\calF_{J+1}, \ldots, \calF_K$ will be carried by $\calF_J$.
Let $v_0, v_{1}, \ldots, v_m$ be the vertices of $\gamma$ in $\cover{\calF}_J$.
Let $e_1, e_{2}, \ldots, e_{m}$ be the edges of $\gamma$ in $\cover{\calF}_J$.
We begin by performing an augmented 2-3 move along $\phi_M(e_1)$, from $\phi_M(v_0)$ to $\phi_M(v_{1})$.
The edge $e_1$ is $L$--flippable by the minimality of $\gamma$ and because in the loosening stage we destroyed all cyclic edges. 

We now proceed recursively. 
Assume that the previous augmented 2-3 move, producing the foam $\calG$ (one of the $\calF_i$), was the first to introduce the region $E$ (and the label $L(E)$) to the zero-handle $\eta'(v_j)$.
We now tunnel through $\eta'(v_j)$ in a process very similar to that described in \refsec{Tunnelling}.
Here are some of the details.
Let $\beta_j$ be a minimal path in the one-skeleton of $\eta'(v_j) \cap \cover{\calG}$ connecting a vertex incident to $E$ with an endpoint of an edge $e'_{j+1}$ that meets $\eta'(e_{j+1})$.
Form $\beta'_j$ by adding $e'_{j+1}$ to the end of $\beta_j$.
We perform augmented 2-3 moves along (the image under the covering map $\phi$ of) each edge of $\beta'_j$ in turn.
The last augmented 2-3 move in our sequence, along $\phi_M(e'_{j+1})$, is the first to introduce the region $E$ (and the label $L(E)$) to the zero-handle $\eta'(v_{j+1})$.

This completes the recursive step.
Once we have processed all edges of $\gamma$, we have produced a foam $\calH$.
The last augmented 2-3 move introduced, for the first time, the region $E$ (and the label $L(E)$) to the zero-handle $\eta'(v_m)$, which also contains at least one vertex of $f$.
Let $\beta_m$ be a minimal path in $\eta'(v_m) \cap \calH$ connecting a vertex incident to $E$ with $f$.
We perform augmented 2-3 moves along each edge of $\beta_m$ in turn.
Again this process is very similar to that described in \refsec{Tunnelling}.

\subsection{Unique contact}
\label{Sec:UniqueContact}

Having processed through all of $\beta_m$ we reach a foam in which the region $E$ is now in contact with $f$.
There are three possible positions for $f$ in \reffig{Augmented2-3Final}.
These are the three sectors at the bottom of the figure.
Suppose that more than one of these is $f$.
In this case, before performing the last augmented 2-3 move we perform some extra V-moves to block off all but one of the sectors from $E$.
That is, we perform V-moves at the terminal end of the edge and then perform the augmented 2-3 move to produce the final foam $\calF_K$.
\reffig{Augmented2-3ProtectE} shows the result when all three sectors are $f$.
The effect is that the region $E$ now meets $f$ along precisely one edge.
We denote this edge by $e_E$.

\begin{figure}[htbp]
\labellist
\hair 2pt \small
\pinlabel $E$ at 127 255
\pinlabel $f$ at 68 125
\pinlabel $e_E$ [t] at 118 105
\endlabellist
\includegraphics[width = 0.53\textwidth]{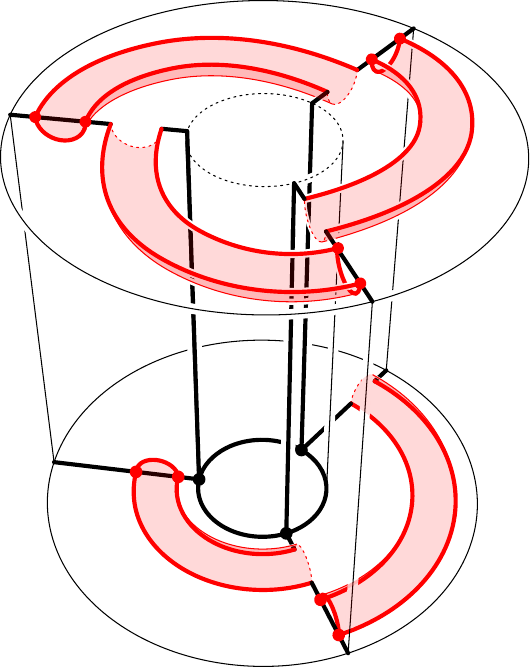}
\caption{We apply extra V-moves to limit the number of edges where $E$ meets $f$ to one, namely the edge $e_E$.}
\label{Fig:Augmented2-3ProtectE}
\end{figure}

\subsection{Verifying hypotheses}

Following the loosening and contacting stages builds us a sequence of 2-3 and 3-2 moves taking us from $\calF_0$ to $\calF_K$ via $L$--essential foams.
Applying \refcor{CommutingDiagram} at each step builds us a sequence of 2-3 and 3-2 moves taking us from $\calF_0[\delta]$ to $\calF_K[\delta]$ via $L$--essential foams.
What remains is to check that we can apply \reflem{Do0-2ManyEdges} to connect $\calF_K$ to $\calF_K[\delta]$.

\begin{lemma}
The foam $\calF_K$ satisfies the hypotheses of \reflem{Do0-2ManyEdges}, taking $E_k$ to be $E$.
\end{lemma}

\begin{proof}
Recall that $s$ is the side of $\delta$ containing $e_E$.
Thus $\phi_M(s)$ contains $\phi_M(e_E)$ and the two nature reserve edges containing the ends of $\phidelta$.
Since $\phi_M(e_E)$ is not a nature reserve edge there are at least three edges on the side $s$.

By \reflem{AugmentedOmnibus}\refitm{Carried} and \reflem{NatureReserveOmnibus}\refitm{NatureCarried}, 
the face $f$ is always carried by its ancestor in $\cover\calF_J$.
This and the minimality of $\gamma$ (\refsec{BringDistant}) ensure that $f$ does not become incident to $E$ (or any other region with the same label as $E$) before the last move of our sequence.
Thus the label $L(E)$ appears exactly once as the label of an outer region for $f$.

Suppose that $D$ is an outer region for the face $f_s$ but not an outer region for $f$. 
Let $e_D$ be the edge of $f_s$ in $\cover{\calF}_K[\delta]$ along which $D$ meets $f_s$.
Thus $\phi_M(e_D)$ is contained in a nature reserve edge $e'$ in $\calF_K$.
(See Figures~\ref{Fig:NatureReserveGreenRed} and~\ref{Fig:ThreeEdges9b} as well as \refrem{FaceFromSide}.)
Since $e'$ is a nature reserve edge, augmented 2-3 moves taking us from $\calF_J$ to $\calF_K$ do not destroy it.
Thus $D$ is (a descendant of) one of the regions incident to the face $f_s$ as it sat in $\cover{\calF}_J[\delta]$.
Our choice of $E$ (made at the start of \refsec{BringDistant}) then implies that $L(D) \neq L(E)$. 
Thus the label $L(E)$ appears exactly once as the label of an outer region for $f_s$.

The edge $e_E$ and the two adjacent edges around $f$ are not cyclic edges since we destroyed all cyclic edges in the loosening step and augmented 2-3 moves do not introduce them. 
(Also, the extra V-moves of \refsec{UniqueContact} do not introduce any cyclic edges.)
Finally, $\phi_M(e_E)$ is distinct from all other edges of $\phi_M(\bdy f)$ because the other two faces of $\calF_K$ incident to $\phi_M(e_E)$ are not equal to $\phi_M(f)$.
This is because these two faces were created in the final augmented 2-3 move; thus they have no ancestors.
\end{proof}

We now apply \reflem{Do0-2ManyEdges} to $\calF_K$.
This gives the desired sequence of moves connecting $\calF_K$ to $\calF_K[\delta]$.

\begin{remark}
\label{Rem:ProvedDistinctEdges}
This completes the proof of \refprop{Make0-2} assuming \refhyp{OneEndpoint}:
that is, that the endpoints of $\phidelta$ lie on distinct edges of $\calF$.
\end{remark}

\subsection{Both endpoints of $\phidelta$ lie on one edge}
\label{Sec:BothEndpointsOneEdge}

We now replace \refhyp{OneEndpoint} by the following.

\begin{hypothesis}
\label{Hyp:TwoEndpoints}
The endpoints of $\phidelta$ lie on the same edge of $\calF$.
\end{hypothesis}

The proof of \refprop{Make0-2} proceeds as before except that when a loosening move is applied along the edge containing both endpoints of $\phidelta$
in the loosening stage we use the pair nature reserve move.
The only piece of the proof remaining is the following.

\begin{proof}[Proof of \reflem{PairNatureReserveMoveDelta}]

We now describe moves to get from $\calG[\delta]$ to $\calG_p[\delta]$ with all foams being $L$--essential.
Here we do not use V-moves to create the three red snakelets and the green snakelet (as we did in the proof of \reflem{SingletonNatureReserveMoveDelta}). 
Instead we apply the special case of \refprop{Make0-2} that we proved under \refhyp{OneEndpoint}; see \refrem{ProvedDistinctEdges}.
That is, we produce a sequence of five $L$--essential foams 
\[
\calG[\delta] = \calG_0,  \quad \calG_1,  \quad \calG_2,  \quad \calG_3 = \calG'[\delta], \quad \calG_4 =  \calG_p[\delta]
\]
with the following properties.
\begin{itemize}
\item
$\calG'$ is as in \reffig{ImpossibleStaircasePairDelta}.
\item
Each foam has an $L$--flippable edge.
\item
$\calG_j[\epsilon_j] = \calG_{j+1}$. 
That is, $\calG_{j+1}$ is obtained from $\calG_j$ by performing a 0-2 move along an arc $\epsilon_j$.
\item
The endpoints of $\epsilon_j$ lie on distinct edges of $\calG_j$.
(This ensures that \refhyp{OneEndpoint} is satisfied.)
\end{itemize}

See \reffig{ImpossibleStaircasePairNatureReserve}.
The arcs $\epsilon_0$ through $\epsilon_2$ are the cores of the red snakelets of $\calG_p$, while $\epsilon_3$ is the core of the green snakelet.
Let $a$, $b$, and $c$ be the three edge-ends in $\calG$ incident to $p$ that are not part of $e$.
We assume that $a$ is not part of the same edge as either $b$ or $c$.
(It is possible that $b$ and $c$ are part of the same cyclic edge.)
We choose the order of $\epsilon_0$ through $\epsilon_2$ as follows.

\begin{enumerate}
\item
\label{Itm:Epsilon0}
Let $\epsilon_0$ be the arc connecting $a$ to $b$.
\item
Let $\epsilon_1$ be the arc connecting $a$ to $c$ (with the end on $a$ between $p$ and the end of $\epsilon_0$).
\item
Finally, let $\epsilon_2$ be the arc connecting $b$ to $c$ (with the end on $c$ between $p$ and the end of $\epsilon_1$, and the end on $b$ not in the segment between $p$ and the end of $\epsilon_0$).
\end{enumerate}

We now check the hypotheses of \refprop{Make0-2} (including \refhyp{OneEndpoint}) for each arc $\epsilon_j$ and for each foam $\calG_j$.

\begin{claim}
\label{Clm:G_0}
 $\calG_0 = \calG[\delta]$ is $L$--essential and has an $L$--flippable edge.
 \end{claim}

\begin{proof} 
Recall that $\calG = \calF_i$ where the latter is a foam created during the loosening stage (\refsec{LooseningStage}).
Under \refhyp{TwoEndpoints} we apply only one pair nature reserve move, which is part of a loosening move taking us from $\calF_i$ to $\calF_{i+1}$.
There are two cases. 
\begin{itemize}
\item Suppose that $i=0$.  
That is, this is the very first loosening move.
Then by the hypotheses of \refprop{Make0-2}, the foam $\calF_0[\delta] = \calF[\delta]$ is $L$--essential and has an $L$--flippable edge.
\item Suppose that $i > 0$.
Then all previous loosening moves were augmented 2-3 moves (since there is only one pair nature reserve move). 
Thus $\calF_i[\delta]$ is $L$--essential and has an $L$--flippable edge by \reflem{AugmentedOmnibus}\refitm{AugmentedLEssential} and \refitm{AugmentedLFlippable}.
\end{itemize}
This proves \refclm{G_0}.
\end{proof}
 
It follows from \refclm{G_0} that $\calG_j$ is $L$--essential for $j > 0$; 
each snakelet creates a new bigon face connecting two regions that were already in contact in $\calG_0$.
The 0-2 moves along $\epsilon_0$ and $\epsilon_1$ are in fact V-moves, so by two applications of \reflem{VMoveL}, the foams $\calG_1$ and $\calG_2$ each have an $L$-flippable edge.
A similar analysis shows that for each of the remaining two 0-2 moves, if an edge not entirely contained in $\calN(p)$ is $L$-flippable then its descendant (that is also not entirely contained in $\calN(p)$) is also $L$-flippable.
(The existence of an $L$--flippable edge on $\calG_3 = \calG_p[\delta]$ also follows from \refrem{Remainder}.)

\refhyp{OneEndpoint} holds for $\epsilon_0$ by construction (see \refitm{Epsilon0} above). 
For both $\epsilon_1$ and $\epsilon_2$, one endpoint is on an edge contained within $\calN(p)$ while the other is not contained within $\calN(p)$.
For $\epsilon_3$, one endpoint is on an edge incident to $p$, while the other is not.

Thus the hypotheses of \refprop{Make0-2} and \refhyp{OneEndpoint} hold for each arc.
\end{proof}

This completes the proof of \refprop{Make0-2}. \qed

\renewcommand{\UrlFont}{\tiny\ttfamily}
\renewcommand\hrefdefaultfont{\tiny\ttfamily}
\bibliographystyle{plainurl}
\bibliography{connecting_essential_triangulations_with_2-3_moves.bib}
\end{document}